\documentclass[11pt,reqno,tbtags]{amsart}

\usepackage{amsthm,amssymb,amsfonts,amsmath,bbm}
\usepackage{amscd} 
\usepackage{mathrsfs}
\usepackage[numbers, sort&compress]{natbib}
\usepackage{caption,graphicx,subcaption}
\usepackage{vmargin}
\usepackage{setspace}
\usepackage{paralist}
\usepackage[pagebackref=true]{hyperref}
\usepackage{color}
\usepackage{stmaryrd} 
\usepackage[refpage,noprefix,intoc]{nomencl} 
\usepackage[usenames,dvipsnames]{xcolor}
\usepackage{soul}
\soulregister\cite7
\soulregister\citep7
\soulregister\citet7
\soulregister\ref7
\soulregister\pageref7

\providecommand{\R}{}

\providecommand{\N}{}

\renewcommand{\R}{\mathbb{R}}

\renewcommand{\N}{{\mathbb N}}

\newcommand{\E}[1]{{\mathbf E}\left[#1\right]}

\newcommand{\p}[1]{{\mathbf P}\left\{#1\right\}}

\newcommand{\brac}[1]{\left[#1\right]}

\newcommand{\Cexp}[2]{\mathbf{E}\brac{\left. #1 \; \right| \; #2}}

 \newcommand{\bag}{\begin{align}}
\newcommand{\bags}{\begin{align*}}
\newcommand{\eag}{\end{align*}}
\newcommand{\eags}{\end{align*}}

\newtheorem{thm}{Theorem}
\newtheorem{lem}[thm]{Lemma}
\newtheorem{prop}[thm]{Proposition}
\newtheorem{cor}[thm]{Corollary}
\newtheorem{dfn}[thm]{Definition}

\newcommand\cC{\mathrm C}

\newcommand\cG{\mathcal G}

\newcommand\cP{\mathcal P}

\newcommand\cR{{\mathcal R}}

\newcommand\cW{\mathcal W}

\newcommand{\rC}{\mathrm{C}}

\newcommand{\rN}{\mathrm{N}}

\newcommand{\rR}{\mathrm{R}} 
 
\newcommand{\rT}{\mathrm{T}} 
\newcommand{\rU}{\mathrm{U}} 
\newcommand{\ru}{\mathrm{u}} 
\newcommand{\rv}{\mathrm{v}}

\newcommand{\rZ}{\mathrm{Z}} 

\newcommand{\bX}{\mathbf{X}}

\newcommand{\bbT}{\mathbb{T}}

\newcommand{\pran}[1]{\left(#1\right)}

\providecommand{\ora}[1]{}
\renewcommand{\ora}[1]{\overrightarrow{#1}}

\newcommand{\eqdist}{\ensuremath{\stackrel{\mathrm{d}}{=}}}

\newcommand{\convas}{\ensuremath{\stackrel{\mathrm{a.s.}}{\rightarrow}}}
\newcommand{\aseq}{\ensuremath{\stackrel{\mathrm{a.s.}}{=}}}

\hypersetup{
    bookmarks=true,         
    unicode=false,          
    pdftoolbar=true,        
    pdfmenubar=true,        
    pdffitwindow=true,      
    pdftitle={My title},    
    pdfauthor={Author},     
    pdfsubject={Subject},   
    pdfnewwindow=true,      
    pdfkeywords={keywords}, 
    colorlinks=true,       
    linkcolor=blue,          
    citecolor=blue,        
    filecolor=blue,      
    urlcolor=blue           
}

\newcommand\urladdrx[1]{{\urladdr{\def~{{\tiny$\sim$}}#1}}}

\DeclareRobustCommand{\SkipTocEntry}[5]{}

\newcommand{\rt}{\mathrm{t}}
\newcommand{\br}{\mathop{\mathrm{br}}}
\newcommand{\leaf}{\mathop{\mathrm{leaf}}}
\renewcommand{\d}{\mathrm{d}}

\DeclareMathOperator{\dir}{Dir} 
\DeclareMathOperator{\ml}{ML} 
\newcommand{\seg}[2]{\left[ \! \left[ #1 , #2 \right] \! \right]}
\newcommand{\oseg}[2]{\left] \! \left] #1 , #2 \right[ \! \right[}
\newcommand{\eqd}{\stackrel{(d)}{=}}
\newcommand{\ind}{1\!\!1}

\begin{document}

\title{Inverting the cut-tree transform} 
\author{Louigi Addario-Berry}
\address{Louigi Addario-Berry, Department of Mathematics and Statistics, McGill University, 805 Sherbrooke Street West, 
		Montr\'eal, Qu\'ebec, H3A 2K6, Canada}
\email{louigi.addario@mcgill.ca}
\urladdrx{http://www.problab.ca/louigi/}
\author{Daphn\'e Dieuleveut}
\address{Daphn\'e Dieuleveut, Equipe de Probabilit\'es, Statistiques et Mod\'elisation, Universit\'e Paris-Sud, Batiment 430, 
91405 Orsay Cedex, France}
\email{daphne.dieuleveut@normalesup.org}
\author{Christina Goldschmidt}
\address{Christina Goldschmidt, Department of Statistics and Lady Margaret Hall, University of Oxford, 24-29 St Giles', Oxford OX1 3LB, UK}
\email{goldschm@stats.ox.ac.uk}
\urladdrx{http://www.stats.ox.ac.uk/~{}goldschm}


\subjclass[2010]{Primary: 60C05. Secondary: 05C05, 60G18, 60G52, 60E07} 

\begin{abstract} 
We consider fragmentations of an $\R$-tree $T$ driven by cuts arriving according to a Poisson process on $T \times [0,\infty)$, where the first co-ordinate specifies the location of the cut and the second the time at which it occurs.  The genealogy of such a fragmentation is encoded by the so-called \emph{cut-tree}, which was introduced by Bertoin and Miermont~\cite{BerMi} for a fragmentation of the Brownian continuum random tree.  The cut-tree was generalised by Dieuleveut~\cite{Dieuleveut} to a fragmentation of the $\alpha$-stable trees, $\alpha \in (1,2)$, and by Broutin and Wang~\cite{BroutinWangptrees} to the inhomogeneous continuum random trees of Aldous and Pitman~\cite{AldousPitmanICRT}.  In the first two cases, the projections of the forest-valued fragmentation processes onto the sequence of masses of their constituent subtrees yields an important family of examples of Bertoin's self-similar fragmentations~\cite{Ber_02}; in the first and third cases the time-reversal of the fragmentation gives an additive coalescent.  Remarkably, in all of these cases, the law of the cut-tree is the same as that of the original $\R$-tree.  

In this paper, we develop a clean general framework for the study of cut-trees of $\R$-trees. We then focus particularly on the problem of \emph{reconstruction}: how to recover the original $\R$-tree from its cut-tree.  This has been studied in the setting of the Brownian CRT by Broutin and Wang~\cite{BroutinWangReverse}, where they prove that it is possible to reconstruct the original tree in distribution. We describe an enrichment of the cut-tree transformation, which endows the cut tree with information we call a \emph{consistent collection of routings}. We show this procedure is well-defined under minimal conditions on the $\R$-trees. We then show that, for the case of the Brownian CRT and the $\alpha$-stable trees with $\alpha \in (1,2)$, the original tree and the Poisson process of cuts thereon can both be almost surely reconstructed from the enriched cut-trees.  For the latter results, our methods make essential use of the self-similarity and re-rooting invariance of these trees.
\end{abstract}

\maketitle


\section{Introduction} 

\subsection{Cutting down trees}

Consider a combinatorial tree $T_n$ with vertices labelled by $1,2, \ldots, n$. A natural cutting operation on $T_n$ consists of picking an edge $\{i,j\}$ uniformly at random and removing it, thus splitting the tree into two subtrees.  Iterating on each of these subtrees, we obtain a discrete fragmentation process on the tree, which continues until the state has been reduced to a forest of isolated vertices. 

A continuum analogue of this process has played an important role in the theory of coalescence and fragmentation.  Let $T$ be a Brownian continuum random tree and consider cuts arriving as a Poisson point process $\cP$ on $T \times [0,\infty)$ of intensity $\lambda \otimes \mathrm{d} t$, where $\lambda$ is the length measure on the skeleton of the tree.  Careful definitions of these objects will be given below; for the moment, we simply note that $\lambda$ is an infinite, but $\sigma$-finite measure, and that there is also a natural \emph{probability} measure $\mu$ on $T$ which allows us to assign masses to its subtrees.  This Poisson cutting of $T$ was first introduced and studied by Aldous and Pitman~\cite{AldPi}. For $s \ge 0$, let $F(s) = (F_1(s), F_2(s), \ldots)$ be the sequence of $\mu$-masses of the connected components of $T \setminus \{p:(p,t) \in \cP, t \le s\}$, listed in decreasing order.  Then $(F(s), s \ge 0)$ is an example of a self-similar fragmentation process, in the terminology of Bertoin~\cite{Ber_02}.  Moreover, a time-reversal of this fragmentation process gives a construction of the standard additive coalescent (see \cite{AldPi} for more details).  

The uniform cutting operation on trees described in the first paragraph was first considered in the mathematical literature in the early 1970's by Meir and Moon~\cite{MeirMoon1, MeirMoon2}.  They applied it repeatedly to the component containing a particular vertex (labelled $1$, say) and investigated the number of cuts required to isolate that vertex.  In \cite{MeirMoon1} and \cite{MeirMoon2}, Meir and Moon focussed on cutting down two particular classes of random trees: uniform random trees, and random recursive trees.  In both cases, these models possess a useful \emph{self-similarity property}: the tree containing the vertex labelled 1 after the first cut is, conditioned on its size, again a tree chosen uniformly from the class in question. 

This work spawned a line of research focussing primarily on the number of cuts, $N_n$, required to isolate the root in various models of random trees; see, for example, \cite{FillKapurPanholzer,Panholzer,DrIkMoeRoe,IksanovMoehle,Janson06,Holmgren}.  Janson~\cite{Janson06} considered the case where $T_n$ is a Galton--Watson tree with critical offspring distribution of finite variance $\sigma^2$, conditioned to have $n$ vertices.  (This includes the case where $T_n$ is a uniform random tree, since this is equivalent, up to a random labelling, to taking the offspring distribution to be Poisson(1).)  It is well known that the scaling limit of $T_n$ in this case is the Brownian continuum random tree $T$ \cite{CRT1,CRT2,CRT3}.  Janson made the striking observation that $N_n/(\sigma \sqrt{n})$ converges in law to a Rayleigh distribution.  The same limit holds for the rescaled distance between two uniformly chosen vertices in $T_n$, and is the law of the distance between two uniformly chosen points in the Brownian continuum random tree.  It was later shown~\cite{BerFires,ABBH,BerMi} that this common limit can be understood by using the cuts to couple $T_n$ with a new tree $T'_n$ in such a way that the number of cuts needed to isolate the vertex labelled $1$ in $T_n$ is the same as the distance between two uniformly-chosen vertices in $T'_n$, and where $T'_n$ also converges to the Brownian continuum random tree when suitably rescaled.  If $T_n$ is a uniform random tree, this can, in fact, be done in such a way that $T_n$ and $T'_n$ have exactly the same distribution for each $n$, using a construction called the Markov chainsaw \cite{ABBH}. The Markov chainsaw takes the sequence of subtrees which are severed from the subtree containing the root, and glues them along a path; one obtains a tree $T_n'$ which has the same distribution as $T_n$ along with two marked points (the extremities of the new path) which are uniform random vertices of $T_n'$. An analogous construction can be performed in the continuum.

\subsection{Fragmentation and cut-trees} \label{subsec:fragncutrees}

In this paper, we will focus on a construction which tracks the whole fragmentation, not just the cuts which affect the component of the root. Consider a discrete tree $T_n$ repeatedly cut at uniformly chosen edges.  The \emph{cut-tree} $C_n$ of $T_n$ represents the genealogy of this fragmentation process.  In this setting, $C_n$ is a binary tree with $n-1$ internal vertices and $n$ leaves, where the leaves correspond to the vertices of $T_n$ and the internal vertices to the non-singleton \emph{blocks} (that is, the collections of labels of the subtrees) appearing at some stage of the fragmentation.  The tree $C_n$ is rooted at a vertex corresponding to $[n]: = \{1,2,\ldots,n\}$, and the two children of the root are the two blocks into which the first cut splits $[n]$. More generally, for a non-singleton block $B \subset [n]$, the two children of $B$ are the two blocks into which the next cut to arrive splits $B$.

\begin{figure}
\centering
\begin{subfigure}[b]{0.3\textwidth}
		\includegraphics[width=\textwidth]{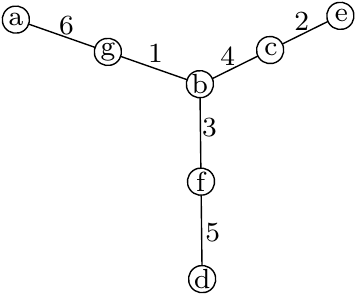}
        \end{subfigure}%
        \quad
\begin{subfigure}[b]{0.3\textwidth}
	\includegraphics[width=\textwidth,page=2]{discretecut.pdf}
        \end{subfigure}%
        \quad
        \begin{subfigure}[b]{0.3\textwidth}

	\includegraphics[width=\textwidth,page=3]{discretecut.pdf}
        \end{subfigure}%
	\caption{Left: A discrete tree $T$ with vertex labels from $\{a,b,c,d,e,f,g\}$. Edges are marked with the time at which they are cut. Center: The resulting cut tree $C$. Each internal vertex corresponds to a cut, and is labelled by the block of the fragmentation process which is split by that cut. Right: Cuts correspond to edges in $T$; here each internal vertex of $C$ is labelled by the pair of endpoints of the corresponding edge.} 
		\label{fig:cut_disc}
\end{figure}

The cut-tree was introduced by Bertoin in \cite{BerFires}, where he considered the case in which $T_n$ is a uniform random tree with $n$ vertices (although the name ``cut-tree'' was first coined subsequently in \cite{BerMi}).  The idea of using a tree to track the genealogy of a discrete fragmentation process also notably appears earlier in \cite{HaMi}.  In \cite{BerMi}, Bertoin and Miermont took $T_n$ to be a Galton--Watson tree with critical offspring distribution of finite variance $\sigma^2$, conditioned to have $n$ vertices.   Bertoin and Miermont proved the following remarkable result. View $T_n$ (resp.\ $C_n$) as a measured metric space by taking $\sigma n^{-1/2}$ (resp.\ $\sigma^{-1}n^{-1/2}$) times the graph distance as the metric, and in both cases endowing the vertices with the uniform probability measure. Then the pair $(T_n,C_n)$ converges in distribution as $n \to \infty$ (in the Gromov--Prokhorov sense) to a pair of dependent Brownian continuum random trees $(T,C)$.  The second CRT is obtained from the first by a continuum analogue of the discrete cut-tree construction discussed above;  we will describe this in detail (and in greater generality) below, once we have introduced the necessary notation.

Cut-trees have been considered for other models of random discrete trees, notably for random recursive trees in \cite{BertoinCutTree, BaurBertoin}.  In that setting, the tree $T_n$ itself (when endowed with the graph distance) does not possess an interesting scaling limit, but the corresponding cut-tree $C_n$, thought of as a metric space using the graph distance divided by $n/\log n$, and endowed with the uniform probability measure on its leaves, converges in the Gromov--Hausdorff--Prokhorov sense to the unit interval endowed with the Lebesgue measure \cite{BertoinCutTree}.  (We observe that there are minor variations in the way that discrete cut-trees are defined in the existing literature. We will gloss over these differences since our primary interest is in the continuous case, where there is no ambiguity of definition.)

The fragmentation of the Brownian continuum random tree via Poisson cutting was generalised to a fragmentation process of the $\alpha$-stable trees $T_{\alpha}$, $\alpha \in (1,2)$, by Miermont in \cite{Mi_05}.  The stable trees are the scaling limits of Galton--Watson trees $T_n$ with critical offspring distribution in the domain of attraction of an $\alpha$-stable law, conditioned to have total progeny $n$.  The heavy-tailed nature of the offspring distribution is reflected in the limit by the fact that the branchpoints (or \emph{nodes}) of the tree all have infinite degree almost surely.  Despite this, the nodes can be given a notion of size (which is made precise using a local time).  In order to obtain a self-similar fragmentation, it is necessary now for the cuts to occur at the nodes of $T_{\alpha}$ rather than along the skeleton.  This is achieved by using a Poisson process whose intensity at a particular node is proportional to the ``size'' of that node.  The cut-tree $C_{\alpha}$ corresponding to this fragmentation was introduced and studied in \cite{Dieuleveut}.  Again, it is the case that if $T_n$ is a conditioned critical Galton--Watson tree with offspring distribution in the domain of attraction of an $\alpha$-stable law, and if $C_n$ is the corresponding discrete cut-tree (suitably adapted to take into account of cutting at vertices rather than edges), then $(T_n, C_n)$ suitably rescaled converges (in the Gromov--Prokhorov sense) to the pair  $(T_{\alpha}, C_{\alpha})$, where $T_{\alpha}$ and $C_{\alpha}$ are (dependent) $\alpha$-stable trees. In the sequel, we will often think of and refer to the Brownian CRT as the $2$-stable tree.

Broutin and Wang~\cite{BroutinWangptrees} generalised the Brownian cut-tree in a different direction, to the inhomogeneous continuum random trees (ICRT's) of Aldous and Pitman~\cite{AldousPitmanICRT}. (These are the scaling limits of the so-called $p$-trees~\cite{CamarriPitman,AldousPitmanRSA}.)  Cutting an ICRT according to the points of a Poisson process, whose intensity measure is now a linear combination of the length measure on the skeleton and a measure on the nodes, again yields a sort of fragmentation process (although no longer, in general, a self-similar or even Markovian one), whose time-reversal is a (non-standard) additive coalescent \cite{AldousPitmanICRT}. Broutin and Wang established that (a particular version of) the cut-tree of a $p$-tree is again a $p$-tree. They also showed that the cut-tree of a (discrete) $p$-tree converges to the continuum cut-tree of the scaling limit ICRT.  Moreover, the cut-tree of an ICRT again has the same law as the original tree.

Abraham and Delmas~\cite{AbrD_FARPLT}, working in the continuum, proved an analogue of the Markov chainsaw result for the L\'evy trees of Duquesne and Le Gall~\cite{DuqLG_02}, which form the general family of scaling limits of conditioned Galton--Watson trees. In particular, they showed that there is a measure $\lambda$ such that if one cuts the tree in a Poisson manner with intensity $\lambda \otimes \mathrm{d} t$ and glues the trees which get separated from the component containing the root along a line-segment then, working under the excursion measure, one again obtains a L\'evy tree with the same ``law'' as the original tree. We understand that the cut-tree of a L\'evy tree is the subject of work in progress by Broutin and Wang.

\subsection{A general framework, and reconstruction} \label{subsec:general}

In this paper, we work directly in the continuum, and establish a general framework for the study of the cut-tree of an $\R$-tree $T$, where the cutting occurs according to a Poisson random measure on $T \times [0,\infty)$ of intensity $\lambda \otimes \mathrm{d} t$, and where $\lambda$ is any measure on $T$ satisfying certain natural conditions.  This encompasses all of the examples which have previously been studied.  In this general setting, we establish conditions under which it is possible to make sense of a unique cut-tree $C$. 
Under a compactness assumption, we are also able to define the push-forwards of probability measures on $T$ to the cut-tree $C$, by studying the push-forwards of empirical measures in $T$.

Our main result concerns the problem of \emph{reconstruction}, that is, recovering the original tree $T$ from its cut-tree $C$. In the Brownian CRT setting, the paper \cite{ABBH} describes a partial cut-tree, in which only cuts of the component containing a root vertex are considered, and describes how to reconstruct $T$ in distribution from this partial cut-tree. This result is generalized by Broutin and Wang in \cite{BroutinWangptrees} to a partial reconstruction result for cut-trees of ICRTs, and in \cite{BroutinWangReverse} to complete reconstruction in the case of the Brownian CRT. More precisely, in \cite{BroutinWangReverse} they describe what they call a ``shuffling'' operation on trees. Writing $s(T)$ for the shuffling of the tree $T$, they show that the pair $(s(T),T)$ has the same law as the pair $(T,C)$, where $C$ is the cut-tree of $T$. Thus, the shuffling operation is a \emph{distributional} inverse of the cut-tree operation, for Brownian CRTs. However, the question of whether the original tree can be recovered almost surely was left open. One of the contributions of this paper is to establish that, indeed, $\alpha$-stable trees can almost surely be reconstructed from their cut-trees, for all $\alpha \in (1,2]$. 

To state our theorems formally requires some technical set-up, which we defer to the next section.  It is, however, instructive to consider the discrete reconstruction problem, since what we do in the continuum will be analogous and the discrete version is rather easier to visualise.  We will focus on the situation where we cut at edges, so that the cut-tree is defined as at the start of Section~\ref{subsec:fragncutrees}.  
Then the extra information which is required in order to reconstruct $T_n$ from $C_n$ is precisely the set of labels of the edges in the original tree.  

Earlier we thought of an internal vertex of $C_n$ as representing a non-singleton block $B$ of the fragmentation, where $B$ contains all of the vertices labelling the leaves in the subtree above that internal vertex.  We may equally think of such an internal vertex as corresponding to the edge $\{i,j\}$ which is cut and causes $B$ to fragment, and from this perspective it is natural to mark this vertex with the pair $\{i,j\}$. In the discrete setting, such markings provide enough information to recover $T_n$; indeed, they fully specify the edge set of $T_n$, so reconstruction is trivial! However, this does not generalise to the $\R$-tree setting. A natural question, and one which we partly answer in this paper, is whether an analogue of ``labelling by cut-edges'' can be defined for the cut-trees of $\R$-trees; and, if so, whether such a labelling contains enough information to allow reconstruction, as in the discrete setting. In the next paragraph, we sketch the reconstruction procedure whose continuum analogue we develop in the sequel.

Suppose that we wish to recover the path between two vertices $i$ and $j$.  Then we may do so as follows.  The subtrees containing $i$ and $j$ were separated by a cut which is represented in $C_n$ by the most recent common ancestor (MRCA) of $i$ and $j$; call this node $i \wedge j$.  The internal node $i \wedge j$ is marked with two labels, $k$ and $\ell$, where the vertex $k$ lies in the subtree above $i \wedge j$ containing $i$ and the vertex $\ell$ lies in the subtree above $i \wedge j$ containing $j$.  We call the pair $(k,\ell)$ a \emph{signpost} for $i$ and $j$.  So we now know that $\{k,\ell\}$ is an edge of $T_n$ lying on the path between $i$ and $j$.  In order to recover the rest of the path, we need to determine the path between $i$ and $k$ and the path between $\ell$ and $j$.  We do this by repeating the same procedure in the subtree above $i \wedge j$ containing $i$ and $k$ and in the subtree above $i \wedge j$ containing $j$ and $\ell$.  So, for example, we find the MRCA of $i$ and $k$, $i \wedge k$, and consider its marks, which tell us about the edge which was cut resulting in the separation of $i$ and $k$ into different subtrees.  We continue this recursively in each subtree, stopping in a particular subtree only when both of the marks on the MRCA are those of vertices already observed on the path.  We will call a \emph{routing} the collection of signposts used in this process.  Although somewhat cumbersome in the discrete setting, this procedure turns out to generalise nicely to the continuum, whereas the notion of edges does not.

\subsection{Stable trees as fixed points}
This work may be viewed in part as a contribution to the literature on transformations with stable trees as a fixed point. The articles \cite{epw,ew} were perhaps the first to explicitly take this perspective; motivated by problems from phylogenetics and algorithmic computational biology, they introduce \emph{cutting and regrafting} operations on CRTs, and show that these operations have the Brownian CRT as fixed points. The main results of \cite{ABBH, BroutinWangptrees, Dieuleveut} state that the Brownian, inhomogeneous, and $\alpha$-stable trees, respectively, are all fixed points of suitable cut-tree operations. We also mention the quite recent work of Albenque and the third author \cite{AG}, which describes a CRT transformation for which the Brownian CRT is the \emph{unique} fixed point, and furthermore shows that the fixed point is attractive. It would be interesting to establish analogous results for cut-tree operations.

\subsection{Outline}

We conclude this rather lengthy introduction with an outline of the remainder of the paper. Section~\ref{sec:notation} formally introduces some of the basic objects and random variables of study, including $\R$-trees and their marked and measured versions, and the $\alpha$-stable trees. Section~\ref{sec:cut_general} presents our general construction of cut-trees of measured $\R$-trees, and of the routing information which we use for reconstruction. 

In Section~\ref{sec:stable_trees} we specialize our attention to stable trees, and show that almost sure reconstruction is possible in this case. In Section~\ref{sec:dist_ident} we establish a fixed-point identity for size-biased Mittag-Leffler random variables. We use this identity in Section~\ref{sec:reconstruct}, together with an endogeny result and a somewhat subtle martingale argument, to show that it is possible to almost surely reconstruct distances between two random points. We extend the reconstruction from two points to all points in Section~\ref{sec:recover_all_dist}. 

Finally, Section~\ref{sec:conc} contains some more speculative remarks, and presents several questions and avenues for future research.

\section{Trees and metric spaces}\label{sec:notation}

Fix a measurable space $(S,\mathcal{S})$ and a finite measure $\mu$ on $(S,\mathcal{S})$.  
we write $X \sim \mu$ if $X$ is an $S$-valued random variable with law $\mu/\mu(S)$. 
For $R \in \mathcal{S}$, we write $\mu_R = \mu(R)^{-1} \mu\vert_R$ for the restriction of $\mu$ to $R$, rescaled to form a probability measure.  

\subsection{$\R$-trees}\label{sec:rtrees}

We begin by recalling some standard definitions. 
For a metric space $(M,d)$ and $S \subset M$, we write $(S,d)$ as shorthand for the metric space $(S,d|_{S \times S})$.  
\begin{dfn}
A metric space $(T,d)$ is an \emph{$\R$-tree} if, for every $u,v \in T$:
\begin{itemize}
\item there exists a unique isometry $f_{u,v}$ from $[0,d(u,v)]$ into $T$ such that $f_{u,v}(0) = u$ and $f_{u,v}(d(u,v)) = v$;
\item for any continuous injective map $f:[0,1] \to T$, such that $f(0)=u$ and $f(1)=v$, we have
\begin{align*}
f([0,1]) = f_{u,v} ([0,d(u,v)]) := \seg{u}{v}.
\end{align*}
\end{itemize}

A rooted $\R$-tree is an $\R$-tree $(T,d,\rho)$ with a distinguished point $\rho$ called the \emph{root}. 
\end{dfn}
Note that we do not require $\R$-trees to be compact.

Let $\rT=(T,d)$ be an $\R$-tree. The \emph{degree} $\mathrm{deg}(x)$ of $x \in T$ is the number of connected components of $T \setminus \{x\}$.  An element $x \in T$ is a \emph{leaf} if it has degree 1; we write $\leaf(\rT)$ for the set of leaves of $T$.  An element $x \in T$ is a \emph{branchpoint} if it has degree at least 3; we write $\br(\rT)$ for the set of branchpoints of $T$. For $x,y \in T$, write $\oseg{x}{y}$ for $\seg{x}{y} \setminus \{x,y\}$.  
The {\em skeleton} of $\rT$, denoted $\mathrm{skel}(\rT)$, is the set $\bigcup_{x,y \in T} (\seg{x}{y} \setminus \{x,y\})$ of vertices with degree at least two. 
We observe that the metric $d$ gives rise to a \emph{length measure} $\lambda$, supported by $\mathrm{skel}(\rT)$, which is the unique $\sigma$-finite measure such that $\lambda(\seg{u}{v}) = d(u,v)$.

A common way to encode a rooted $\R$-tree is via an excursion, that is, a continuous function $h: [0,1] \to \R_+$ such that $h(0) = h(1) = 0$ and $h(x) > 0$ for $x \in (0,1)$.  For $x,y \in [0,1]$, let
\[
d(x,y) = h(x) + h(y) - 2 \min_{x \wedge y \le z \le x \vee y} h(z),
\]
and define an equivalence relation $\sim$ by declaring $x \sim y$ if $d(x,y) = 0$.  Let $T = [0,1]/\sim$.  Then it can be checked that $(T, d)$ is an $\R$-tree which may be naturally rooted at the equivalence class $\rho$ of 0, and endowed with the measure $\mu$ which is the push-forward of the Lebesgue measure on $[0,1]$ onto $T$.

A rooted $\R$-tree $\rT=(T,d,\rho)$ comes with a genealogical order $\prec$ such that $x \prec y$ if and only if $x \in \seg{\rho}{y}$ and $x \neq y$; we say that $x$ is an \emph{ancestor} of $y$.  The \emph{most recent common ancestor} $x\wedge y$ of $x,y \in T$ is the point of $\{z: z \prec x, z \prec y\}$ which maximises $d(\rho,z)$.
For $z \in T$, write $T_z = \{x \in T: x \wedge z=z\}$, and $\rT_z=(T_z,d,z)$ for the subtree above $z$. Next, for $y \in T_z$ write $T_z^y=\{x \in T: z \prec x \wedge y\}\cup \{z\}$ and $\rT_z^y = (T_z^y,d,z)$ 
for the subtree above $z$ containing $y$. Note that $T_z^y \subset T_z$. 

A \emph{measured} $\R$-tree is a triple $(T,d,\mu)$, where $(T,d)$ is an $\R$-tree and $\mu$ is a Borel probability measure on $T$. A \emph{pointed} $\R$-tree is a triple $(T,d,s)$, where $s$ is a finite or infinite sequence of elements of $T$. We combine adjectives in the natural way; thus, for example, a rooted measured pointed $\R$-tree is a quintuple $(T,d,\rho,\mu,s)$.

\subsection{Random $\R$-trees}\label{sec:rrt}
The topological prerequisites for the study of random $\R$-trees have been addressed by several authors \cite{adh,bbi,epw,ew,fukaya,GPW_08,gromov07,miermont}. Many of the aforementioned papers study random metric spaces more generally, but the theory specializes nicely to the setting of $\R$-trees, which is all we require in the current work. 
In this section we summarize the definitions and results which we require. 

We hereafter restrict our attention to complete, locally compact $\R$-trees. Fix rooted measured pointed $\R$-trees $\rT=(T,d,\rho,\mu,s)$ and $\rT' = (T',d',\rho',\mu',s')$. We say $\rT$ and $\rT'$ are \emph{isometric} if 
there exists a metric space isometry $\phi:T \to T'$ which sends $\rho$ to $\rho'$, $\mu$ to $\mu'$, and $s$ to $s'$. More precisely, $s$ and $s'$ must have the same cardinality, and $\phi$ must satisfy the following:
\begin{itemize}
\item $\phi(\rho)=\rho'$;
\item $\mu'$ is the pushforward of $\mu$ under $\phi$;
\item writing $s_k$ and $s_k'$ for the $k$'th elements of $s$ and $s'$, respectively, then $\phi(s_k)=s_k'$ for all $k$.
\end{itemize}
We define isometry for less adjective-heavy $\R$-trees by relaxing the constraints on $\phi$ correspondingly. For example, measured $\R$-trees $(T,d,\mu)$ and $(T',d',\mu')$ are isometric if there exists a metric space isometry $\phi:T \to T'$ which sends $\mu$ to $\mu'$.

Let $\bbT$ be the set of isometry equivalence classes of (complete, locally compact) rooted measured $\R$-trees. Endowing $\bbT$ with the Gromov--Hausdorff--Prokhorov (GHP) distance turns $\bbT$ into a Polish space \cite{adh}, which allows us to consider $\bbT$-valued random variables.

The GHP convergence theory for rooted measured pointed $\R$-trees with a finite number of marked points 
is described in Section 6 of \cite{miermont}. In order to apply this theory in the current setting, two comments are in order. First, the theory is described for compact. rather than locally compact spaces. However, the  development proceeds identically for locally compact spaces, so we omit the details. Second, in the present paper, we will in fact have a countably infinite number of marks. We briefly comment on the additional, rather standard, topological considerations. For each $1 \le n \le \infty$, fix $\rT_n=(T_n,d_n,\mu_n) \in \bbT$ and let $s_n=(s_{n,i},i \ge 1)$ be a sequence of elements of $T_n$. We say the sequence of marked spaces $(T_n,d_n,\mu_n,s_n)$ converges to $(T_\infty,d_\infty,\mu_\infty,s_\infty)$ if, for each $m \in \N$, $(T_n,d_n,\mu_n,(s_{n,i},i \le m))$ converges to $(T_{\infty},d_{\infty},\mu_{\infty},(s_{\infty,i},i \le m))$ in the sense described in \cite{miermont}.

We conclude by noting a sufficient condition for two $\R$-trees to have the same law. Let $\rT = (T,d,\rho,\mu)$ and $\rT' = (T',d',\rho',\mu')$ be $\bbT$-valued random variables. Conditional on $\rT$, let $(U_i)_{i \in \N}$ be a sequence of i.i.d.\ points of $T$ with common law $\mu$ and, conditional on $\rT'$, let $(U'_i)_{i \in \N}$ be a sequence of i.i.d.\ points of $T'$ with common law $\mu'$.  (See Section 6.5 of \cite{miermont} for a treatment of the measurability issues involved in randomly sampling from a random metric space.) Letting $U_0 = \rho$ and $U'_0 = \rho'$, if
\[
(d(U_i, U_j))_{i, j \ge 0} \eqd (d'(U'_i, U'_j))_{i,j \ge 0},
\]
then $\rT$ and $\rT'$ are identically distributed (see \cite{gromov07}, Theorem $3 \frac{1}{2}.5$).

\subsection{The stable trees}\label{sec:stable_def}

A \emph{stable tree} of index $\alpha \in (1,2]$ is a random measured $\R$-tree $\rT=(T,d_T,\mu)$ derived from a suitably normalized excursion of length $1$ of a spectrally positive $\alpha$-stable L\'evy process. Equivalently, it is the scaling limit of large conditioned Galton-Watson trees whose offspring distribution is critical and lies in the domain of attraction of an $\alpha$-stable law. $\rT$ is naturally equipped with a root, $\rho$, which arises as the equivalence class of 0, although we will often not need it and so omit it from the notation.  The following theorem says that we can always regenerate it at will, as it is a uniform pick from the mass measure $\mu$.

\begin{thm} \label{thm:re-rootinvariance}
Let $(T,d_T, \rho, \mu)$ be a rooted stable tree of index $\alpha \in (1,2]$.  Let $r$ be sampled from $T$ according to $\mu$.  Then $(T,d_T,r, \mu)$ has the same (unconditional) distribution as $(T,d_T,\rho,\mu)$.
\end{thm}

This follows from Aldous~\cite{CRT1} for $\alpha=2$ and Proposition 4.8 of Duquesne and Le Gall~\cite{DuqLG_05} for $\alpha \in (1,2)$. 

The stable tree of index $\alpha=2$ corresponds to the Brownian continuum random tree encoded by $(\sqrt{2} \mathbbm{e}(t))_{0 \leq t \leq 1}$, where $\mathbbm{e}$ denotes a normalized Brownian excursion. A significant difference between the Brownian CRT and the stable trees of index $\alpha \in (1,2)$ is the fact that the Brownian CRT is almost surely binary (i.e. $\mathrm{deg}(b)=3$ almost surely for all branchpoints $b$), whereas a stable tree of index $\alpha \in (1,2)$ almost surely has only branchpoints of infinite degree.  In the latter case, the ``size'' of a branchpoint $b \in \br(\rT)$ can be described by the quantity
\begin{gather*}
A (b) = \lim_{\epsilon \rightarrow 0^+} \frac{1}{\epsilon} \mu \{ v \in T: b \in \seg{\rho}{v}, d_{T}(b,v) < \epsilon \},
\end{gather*}
whose existence was proved in \cite{Mi_05} (see also \cite{DuqLG_05}). It is useful to define a measure $\Lambda$ on $T$, as follows. If $\alpha=2$ then $\Lambda$ is twice the length measure on $\mathrm{skel}(\rT)$, and if $\alpha \in (1,2)$ then $\Lambda = \sum_{b \in \br(T)} A(b)\cdot \delta_b$. 

The next theorem concerns the self-similarity of the stable trees, and will play an important role in Section~\ref{sec:stable_trees}. Let $x,y,z$ be independent points of $T$ with common law $\mu$, and let $b$ be the common branchpoint of $x,y,z$ (i.e. the unique element of $\seg{x}{y} \cap \seg{x}{z} \cap \seg{y}{z}$). Recall that $T^x_b$ is the subtree of $T$ consisting of all points $w$ with $b \not\in \oseg{x}{w}$, and write $\rT^x$ and $\rT^y$ for the measured $\R$-trees induced by $(T^x_b, \mu(T_b^x)^{-1+1/\alpha} \cdot d, b, \mu_{T_b^x})$ and $(T^y_b, \mu(T_b^y)^{-1+1/\alpha}\cdot d, b, \mu_{T_b^y})$ respectively.

\begin{thm}  \label{T:cut_in_3}
\begin{enumerate}
\item The trees $\rT^x$ and $\rT^y$ are independent $\alpha$-stable trees, independent of the vector $(\mu(T_b^x), \mu(T_b^y))$.
\item Conditionally on $\rT^x$ (resp. $\rT^y$), the points $x$ and $b$ (resp.\ $y$ and $b$) are independent points of $\rT^x$ (resp. $\rT^y$) sampled according to its rescaled mass measure. 
\end{enumerate}
\end{thm}

\begin{proof}
See Theorem 2 of Aldous~\cite{Ald_RecSelfSim} for the case $\alpha = 2$ and Corollary 10 of Haas, Pitman and Winkel~\cite{HaasPitmanWinkel} for $\alpha \in (1,2)$.
\end{proof}

We refer the reader to \cite{DuqLG_02, DuqLG_05,Mi_03,Mi_05} for more on the theory of stable trees. 

\section{The cut-tree of an $\R$-tree: general theory}\label{sec:cut_general}

\subsection{Defining branch lengths for the cut-tree}\label{sec:cut_tree_def}
Throughout Section~\ref{sec:cut_general} we let $\rT=(T,d_T,\mu)$ be an $\R$-tree with $\mu(T)=\mu(\leaf(T))=1$. We further fix a $\sigma$-finite Borel measure $\lambda$ on $T$ with $\lambda(\leaf(\rT))=0$, and let $\cP = ((p_i,t_i),i \in I)$ be a Poisson point process on $T\times[0,\infty)$ with intensity measure $\lambda \otimes \d t$.

We view each point $p_i$ as a {\em cut}, which arrives at time $t_i$. 
For all $t \ge 0$ and $x \in T\setminus \{p_i : t_i \le t\}$, let $T(x,t)$ be the connected component of $T \setminus \{p_i : t_i \le t\}$ containing $x$, and let $\rT(x,t)$ be the corresponding $\R$-tree.  For $x \in \{p_i : t_i \le t\}$, let $\rT(x,t)$ be the subtree of $\rT$ containing only the point $x$. 

For distinct points $x,y\in T$, let $(t(x,y),p(x,y))$ be the point of $\cP$ which first ``strictly separates'' $x$ and $y$, so $p(x,y) \in \oseg{x}{y}$ and $t(x,y)$ is minimal subject to this. Also, set $t(x,x)=\infty$. 
Next, for $S \subset T$, let $t(S) = \inf\{t(x,y): x,y \in S\}$ be the first time a cut separates two elements of $S$. 

Fix $x \in T$. For $t \ge 0$, we define 
\[
\ell(x,t) = \int_0^t \mu(T(x,s)) \d s, 
\]
and $\ell(x)=\ell(x,\infty)$. 
Then, for $x,y \in T$, let 
\[
D(x,y) = \ell(x)+\ell(y) - 2\ell(x,t(x,y)) = \ell(x)+\ell(y)-2\ell(y,t(x,y))\, .
\]
Clearly $D(x,y)=D(y,x)$. 

Note that $t(x,y)$ is exponentially distributed with 
parameter $\lambda(\seg{x}{y})$. Since $\lambda$ is 
$\sigma$-finite, $\lambda(\seg{x}{y})$  is finite, so $D(x,y)$ is a.s.\ positive. 
Nonetheless, it is possible that $D(x,y)=0$ for some pairs $x,y$ with $x \ne y$. However, $D$ a.s.\ defines a pseudo-metric. 
\begin{prop}
$D$ satisfies the triangle inequality. 
\end{prop}
\begin{proof}
Fix $x,y,z \in T$. We have $t(x,z) \ge \min(t(x,y),t(y,z))$, so assume without loss of generality that $t(x,z) \ge t(x,y)$. We then have 
\[
D(x,y)+D(y,z) = D(x,z) + 2[\ell(y)-\ell(y,t(y,z)) + \ell(x,t(x,z))-\ell(x,t(x,y))]. 
\]
The quantity in square brackets is non-negative since $\ell$ is non-decreasing in its second argument.
\end{proof}
Now fix a sequence $\ru=(u_i,i \ge 1)$ of distinct points of $T$. 
The next proposition describes a tree encoding the genealogical structure that $\cP$, viewed as a cutting (or fragmentation) process, induces on the elements of $\ru$. 

\begin{prop}\label{prop:build_cut}
Suppose that almost surely $\ell(u_i) < \infty$ for all $i \ge 1$. 
Then almost surely, up to isometry-equivalence there is a unique $\R$-tree 
$\cC^\circ=\cC^\circ(\rT,\cP,\ru):= (C^\circ,d^\circ,\rho)$ containing points $\{\rho\} \cup \N$ and satisfying the following properties. 
\begin{enumerate}
\item The points $C^\circ$ of $\cC^\circ$ satisfy
$C^\circ = \bigcup_{i\in \N} \seg{\rho}{i}$.
\item For all $i,j \in \N$, $d^\circ(\rho,i)=\ell(u_i)$ and $d^\circ(i,j)=D(u_i,u_j)$. 
\end{enumerate}
\end{prop}

\begin{proof}
First, the integrability condition implies that, almost surely, 
$d^\circ(i,j) < \infty$ for all $i,j \in \N \cup \{\rho\}$. For $n \in \N$ we write $[n]=\{1,\ldots,n\}$. 

We next show that for any metric $D$ on $\N\cup \{\rho\}$, up to isometry there is at {\em most} one $\R$-tree containing $\N$
whose restriction to $\N$ is isometric to $(\N,D)$. Indeed, suppose $\rR=(R,d_R)$ and $\rR'=(R',d_{R'})$ are two such trees. Then for all $n \in \N$, the subtrees of $R$ and $R'$ induced by $\bigcup_{i,j \in [n] \cup \{\rho\}} \seg{i}{j}$ are easily seen to be isometric. Further, any isometry between them induces an isometry of the subtrees spanned by $\bigcup_{i,j \in [k] \cup \{\rho\}} \seg{i}{j}$, for any $k< n$. We may thus take a projective limit to obtain an isometry between $\rR$ and $\rR'$. 

It remains to prove existence, which in fact follows in much the same way once we verify that the distances specified by $D$ are ``tree-like''. More precisely, suppose that for each $n \in \N$, there exists an $\R$-tree $\rR_n=(R_n,d_n)$ containing $[n]\cup \{\rho\}$ such that $D(i,j)=d_n(i,j)$ for all $i,j \in [n]\cup \{\rho\}$ and such that $R_n = \bigcup_{i\le n} \seg{\rho}{i}$. Then a projective limit of the sequence $\rR_n$ has the required properties. 

Finally, for any $t \ge 0$, the collection of cuts $\{p_i: t_i \le t\}$ induces a partition of $[n]$: for $j,k \in [n]$, $j$ and $k$ lie in the same part at time $t$ if $t(u_j,u_k) > t$ and neither $u_j$ nor $u_k$ is an element of $\{p_i: t_i \le t\}$. This partition-valued process has an evident genealogical structure, and so describes a rooted discrete tree $F$ with leaves $[n]$. 
In this picture, $\rho$ is simply the root of $F$. Note that $\rho$ has degree one; see Figure~\ref{fig:cut_gen}.

\begin{figure}
\centering
\begin{subfigure}[b]{0.3\textwidth}
		\includegraphics[width=\textwidth]{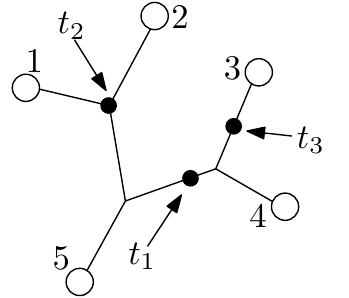}
        \end{subfigure}%
        \quad
\begin{subfigure}[b]{0.3\textwidth}
	\includegraphics[width=\textwidth,page=2]{partitionstructure.pdf}
	\label{fig:gen_tree}
        \end{subfigure}%
	\caption{Left: the subtree of $\rT$ spanned by vertices $u_1,\ldots,u_5$. Cuts on this subtree arrive at times $t_1 < t_2 < t_3$. Right: the resulting genealogical tree $F$ with leaves $1,\ldots,5$. The figure indicates the correspondence between cuts and branchpoints of $F$.}	\label{fig:cut_gen}
\end{figure}
To each internal node $v$ of $F$, let $L(v)$ be the set of leaves which are descendants of $v$. For each edge $vw$ of $F$, with $v$ an internal node and $w$ a child of $v$, give $vw$ length $t(\{u_j,j \in L(w)\})-t(\{u_j,j \in L(v)\})$. Finally, the child of $\rho$ is the unique internal node $v$ with $L(v)=[n]$; let the edge $\rho v$ have length $t(\{u_j,j \in L(v)\})=t(\{u_j,j \in [n]\})$. 
The resulting $\R$-tree has the correct distance between any pair $a,b\in [n]\cup \{\rho\}$, and is spanned by the paths between such pairs; this completes the proof of existence. 
\end{proof}

It deserves emphasis that the elements of $\N$ are {\em random} points of the random tree $\cC^\circ$: it is not possible to recover their locations from $\cC^\circ$ alone. 
In fact an analogue of the preceding proposition also holds in the setting where $\ell(x)$ is not $\mu$-a.e.\ finite. Though we do not require this case in the current work, its elaboration introduces several ideas we do use, so we now describe it. For $t \ge 0$, let 
\[
D_t(x,y) = \ell(x,t)+\ell(y,t) - 2\ell(x,t \wedge t(x,y)). 
\]
This is almost surely finite for any fixed $x,y \in T$, and $D_t(x,y) \uparrow D(x,y)$ as $t \to \infty$. Following the proof of Proposition~\ref{prop:build_cut} shows that there is a unique (up to isometry-equivalence) $\R$-tree $\cC_t=(C_t,d_t,\rho)$ satisfying the obvious modifications of conditions (1) and (2). Note, however, that for each $j \in \N$ there are a.s.\ infinitely many $k \in \N$ with $D_t(j,k)=0$. 

The trees $\cC_t=(C_t,d_t,\rho)$ are increasing in the sense that for $t < t'$, $\cC_t$ may be realised as a subtree of $\cC_{t'}$. We may therefore define $\cC^\circ$ as the increasing limit of the process $(\cC_t,t \ge 0)$. 
This definition agrees with that of Proposition~\ref{prop:build_cut} when $\ell(x)$ is almost surely finite for $\mu$-a.e.\ $x$. It additionally endows $\mathrm{skel}(\cC^{\circ})$ with a labelling by ``arrival time'': for $x \in \mathrm{skel}(\cC^{\circ})$, let $\alpha(x) = \inf\{t: x \in C_t\}$. 

To see that this is a measurable quantity, first note that for fixed $j \in \N$, 
the geodesic $\seg{\rho}{j}\setminus\{j\}$ is isometric to the line segment $[0,\ell(u_j))$. The function $m:[0,\infty) \to [0,1)$ given by $m(s) = \mu(T(u_j,s))$ is clearly measurable.
Thus, for $x \in \seg{\rho}{j} \setminus \{j\}$, let
\begin{equation}\label{eq:alpha_formula}
\alpha(x) = \inf\left\{t: \int_0^t m(s)\mathrm{d}s \ge d^\circ(\rho,x) \right\} = \inf\{t: \ell(u_j,t) \ge d^\circ(\rho,x)\}.
\end{equation}
It is easily seen that these labelings are consistent in that the label $\alpha(x)$ does not 
depend on the choice of $j$ with $x \in \seg{\rho}{j}\setminus\{j\}$. 

Write $(C,d,\rho)$ for the completion of $\cC^{\circ}=(C^\circ,d^\circ,\rho)$. 
Note that the elements of $\N$ are points of $\cC^{\circ}$ and thus of $(C,d,\rho)$. For $n \in \N$, let 
$\cC(n)$ be the subtree of $\cC^\circ$ spanned by $\{\rho\} \cup [n]$, so having points $\bigcup_{j \le n} \seg{\rho}{j}$. This is essentially the tree $\rR_n$ from within the proof of Proposition~\ref{prop:build_cut}. 

\subsection{Measures on the cut-tree}\label{sec:measures} 
We next define, for each $n \in \N$, a measure $\nu_n$ on $(C,d,\rho)$ whose support is $C(n)$. 
Let
\[
\cP_n = \{(p_i,t_i) \in \cP~:~ \exists j \le n\mbox{ such that } \mu(T(u_j,t_i)) < \mu(T(u_j,t_i^-))\}. 
\]
This is the set of points whose cuts reduce the mass of the subtree containing some point $\{u_j, j \le n\}$. For $s \ge 0$ let 
$\cP_n(s) = \{(p,t) \in \cP_n: t \le s\}$. We likewise define $\cP(s) = \{(p,t) \in \cP: t \le s\}$. 
For the remainder of Section~\ref{sec:cut_general} we assume that the fragmentation induced by $\cP$ conserves mass in that for all $s > 0$, 
\[
\sum_{\text{$T'$ a component of $T\setminus \{p_i: t_i \in \cP(s)$}\}} \mu(T')=1. 
\]

Denote the set of open connected components of $T\setminus \{p_i: (p_i,t_i) \in \cP_n\}$ by 
$\{T_i: i \in I_n\}$. 
We observe that if $\ell(u_j) < \infty$ almost surely for $j \in \N$ then almost surely no component of $\{T_i,i \in I_n\}$ contains an element of $\{u_j,j \le n\}$. 

For each $i \in I_n$, let 
\[
\sigma_i = \inf\{s \ge 0: T_i~\text{is a connected component of}~T\setminus \cP_n(s)\} \, 
\]
be the creation time of $T_i$. Let 
\[
m_i = \min \{j \le n: u_j~\text{and}~T_i~\text{lie in the same component of}~ 
T\setminus\cP_n(\sigma_i-)\}\, 
\]
be the index of the last point to separate from $T_i$, breaking ties by taking the smallest such. Then let 
$x_i$ be the unique point of $C(n)$ on $\seg{\rho}{m_i}$ satisfying $\alpha(x_i)=\sigma_i$. 

Now define a measure $\nu_n$ on $(C,d,\rho)$ with support $C(n)$ by 
\[
\nu_n = \sum_{i \in I_n} \mu(T_i) \delta_{x_i}\, .
\]
We may view the tree $T_i$ as ``frozen'' at time $\sigma_i$ and attached to $C(n)$ at point $x_i$. With this perspective, $\nu_n$ is obtained by projecting the masses of the frozen subtrees onto their attachment points in $C(n)$. We do not explicitly need this construction, however, so do not formalize it. 
\begin{prop}
If $(C,d,\rho)$ is compact then $\nu_n$ is a Cauchy sequence in the space of Borel measures on $C$, so has a weak limit $\nu$. 
\end{prop}
\begin{proof}
Fix $m > n$ and write $\nu_n = \sum_{i \in I_n} \mu(T_i) \delta_{x_i}$ as above. 
Note that $\cP_n$ is increasing in $n$, so we may view $\{T_j,j \in I_m\}$ as a ``refinement'' of $\{T_i,i \in I_n\}$ in the sense that each tree $T_i$ in the latter set is split by the cuts associated with points of $\cP_m\setminus \cP_n$ into a collection of trees $\{T_{i,j}, j \in J_i\}$ which all lie in the former. Furthermore, we exhaust $\{T_j,j \in I_m\}$ in this manner, in that $I_m = \bigcup_{i \in I_n} J_i$. Finally, since $\lambda(\text{leaf}(T))=0$ and $\mu(T)=\mu(\text{leaf}(T))=1$, we also have $\mu(T_i) = \sum_{j \in J_i} \mu(T_{i,j})$. 

We write 
\[
\nu_m = \sum_{i \in I_n} \sum_{j \in J_i} \mu(T_{i,j}) \delta_{x_{i,j}} 
\]
where, by analogy with the above, $x_{i,j}$ is the point of attachment of $T_{i,j}$ to $C(m)$. 
Now note that $x_{i,j} \in C_{x_i}$ by construction. Since the fragmentation conserves mass, it follows that for all $i \in I_n$, 
\[
\mu(T_i) = \sum_{x_{i,j} \in C_{x_i}} \mu(T_{i,j}). 
\]
We may thus obtain $\nu_n$ from $\nu_m$ by projecting all mass of $\nu_m$ onto the closest point of $C(n)$, so $d_{P}(\nu_n,\nu_m) \le d_H(C(n),C(m))$. Since $C(m)$ contains $C(n)$ and is increasing in $m$, this implies that $d_{P}(\nu_n,\nu_m) \le d_H(C(n),C)$, and the final quantity tends to zero as $n \to \infty$ by compactness. 
\end{proof}
In the case when $(C,d,\rho)$ is compact, we write 
$\cC=(C,d,\rho,\nu)$, and call $\cC=\cC(\rT,\cP,\ru)$ the cut-tree of 
$(\rT,\cP,\ru)$, or sometimes simply the cut-tree of $\rT$. 

The tree $\rT$ comes endowed with measure $\mu$. If it happens that 
\[
\mu = \lim_{n \to \infty} \frac{1}{n} \sum_{j \le n} \delta_{u_j}\, , 
\]
then we say that $\mu$ is the empirical measure of the sequence $\ru$. 
In particular, if the elements of $\ru$ are i.i.d.\ with law $\mu$ then this holds by the Glivenko-Cantelli theorem. 
\begin{prop}\label{prop:empirical_measure}
If $\mu$ is the empirical measure of $\ru$ and $(C,d,\rho)$ is compact then $\nu$ is the empirical measure of $\N \subset C$, i.e.
\[
\nu = \lim_{n \to \infty} \frac{1}{n} \sum_{j \le n} \delta_j\, .
\]
\end{prop}
\begin{proof}
For $n < m$ and for $i \in I_n$ let 
\[
\hat{\mu}_{n,m}(T_i) = \frac{1}{m-n} \#\{n < j \le m: u_j \in T_i\}. 
\]
Note that for $i \in I_n$, if $u_j \in T_i$ then 
$x_i \in \seg{\rho}{j}$, i.e., $j$ lies in a subtree of $C$ which is attached to $C(n)$ at the point $x_i$. We thus have 
\[
\hat{\mu}_{n,m}(T_i) = \frac{1}{m-n} \#\{n < j \le m: x_i \in \seg{\rho}{j}\}. 
\]

Next, since $\mu$ is the empirical measure of $\ru$, for fixed $n$ we have
\[
\lim_{m \to \infty} \hat{\mu}_{n,m}(T_i) 
= 
\lim_{m \to \infty} \frac{1}{m} \#\{1 \le j \le m: u_j \in T_i\} 
= \mu(T_i). 
\]
Now let $\nu_{n,m} = \sum_{i \in I_n} \hat{\mu}_{n,m}(T_i) \delta_{x_i}$. 
Since $\nu_n$ is a probability measure, 
it follows that 
\[
\nu_n = \lim_{m \to \infty} \nu_{n,m}. 
\]

Finally, writing 
\[
\hat{\nu}_{n,m} = \frac{1}{n-m} \sum_{n < j \le m} \delta_j,
\]
we have 
\[
\lim_{n \to \infty} \sup_{m > n} d_P(\hat{\nu}_{n,m},
\nu_{n,m})
\le \lim_{n \to \infty} d_H(C(n),C) = 0, 
\]
the last equality holding by compactness. This yields
\[
\lim_{n \to \infty} \nu_n = 
\lim_{n \to \infty}\lim_{m \to \infty} \nu_{n,m} = 
\lim_{n \to \infty}\lim_{m \to \infty} \hat{\nu}_{n,m} = 
\lim_{m \to \infty} \frac{1}{m} \sum_{1 \le j \le m} \delta_j\, ,
\]
as required. 
\end{proof}

\subsection{Images in $C$ of points of $T$.}\label{sec:image}
In this subsection we assume $\rT$, the intensity measure $\lambda\otimes \d t$ and the sequence $\ru$ of points of $T$ are such that, almost surely, 
$\ell(u_i) < \infty$ for all $i \ge 1$. By Proposition~\ref{prop:build_cut} we may then define $\cC^{\circ}(\rT,\cP,\ru)$ and its completion $(C,d,\nu)$. 
We further assume that $\N$ is dense in $\cC^{\circ}(\rT,\cP,\ru)$, and therefore in $C$. 

For $k \in \N$ it is natural to associate the point $u_k\in T$ with the point $k \in C$; we call $k$ the {\em image} of $u_k$. 
We now define the image (or {\em images}) in $C$ of a fixed point $x \in T$. 

First suppose that $x$ is not one of the cut points $p_i$; 
this holds a.s.\ precisely if $\lambda(\{x\}) = 0$. 
Let $\ru'=(x,u_1,u_2,\ldots)$. We write $k$ for the image of $u_k$ in $\cC^\circ(\rT,\cP,\ru')$, and $x'$ for the image of $x$. 

We now describe how to identify $x'$ with a point of $C$ that is added during completion from $\cC^\circ(\rT,\cP,\ru)$. We identify $\cC^\circ(\rT,\cP,\ru)$ with the subtree of $\cC^\circ(\rT,\cP,\ru')$ spanned by $\N\cup \{\rho\}$. Since $\N$ is a.s.\ dense in $\cC^\circ(\rT,\cP,\ru')$, we may find a sequence $(i_k,k \ge 1)$ of natural numbers such that in $\cC^\circ(\rT,\cP,\ru')$ we a.s.\ have $i_k \to x'$ as $k \to \infty$. This is also a Cauchy sequence in $\cC^\circ(\rT,\cP,\ru)$, so has a limit in $(C,d,\rho)$. We identify this limit with $x'$; it is the image of $x$ in $(C,d,\rho)$.  It is important that we may then view $\cC^\circ(\rT,\cP,\ru')$ as a subtree of $(C,d,\rho)$ and that, with this perspective, $(C,d,\rho)$ is also the completion of $\cC^\circ(\rT,\cP,\ru')$. 

Next suppose $x$ is one of the cut points $p_\ell$, so $(x,t_\ell) \in \cP$. 
In this case, there is an image of $x$ corresponding to each connected component of $T\setminus \{x\}$. We describe how to find the image 
of $x$ corresponding to a fixed such component $\hat{T}$. 

The idea is to view the cut-tree process as acting on $\hat{T}$, starting at time 
$t_\ell$, and thereby identify the image of $x$. The key point is that the cut-tree behaves nicely under restriction, in a sense we now explain. 
Let 
\[
\hat{\cP} = \{ (p_i,t_i) \in \cP: p_i \in T(x,t_\ell-)\cap \hat{T}, t_i > t_\ell\}\, . 
\]
The set $T(x,t_\ell-)\cap \hat{T}$ is the subtree separated at time $t_\ell$ which is contained in $\hat{T}$, and $\hat{\cP}$ is the set of Poisson points falling 
in this subtree after time $t_{\ell}$. 

Next, let $\rv$ be the set of points $u_i \in T(x,t_\ell-)\cap \hat{T}$; 
for concreteness we list these in increasing order of index as 
$\rv=(v_i,i \ge 1)$. Let $T' = \{x\} \cup (T(x,t_\ell-)\cap \hat{T})$, and let $\rT'=(T',d_{T'})$ be the subtree of $\rT$ induced by $T'$. 

Let $\cP' = \{(p,t-t_\ell): (p,t) \in \hat{\cP}\}$, 
and observe that $\cP'$ is a Poisson process on $T' \times [0,\infty)$ 
with intensity measure $\lambda|_{T'\setminus \{x\}} \otimes \d t$. 
The completion of the tree $\cC^\circ(\rT',\cP',\rv)$ is now isometric to a subtree $C'$ of $(C,d,\rho)$. 
Furthermore, this isometry is uniquely specified by requiring that the images of the points of $\rv$ agree with the images of the corresponding elements of the sequence $\ru$. Since $x$ is not hit by $\cP'$, it also has an image $x'$
in $\cC^\circ(\rT',\cP',\rv)$, which we view as contained in $\cC^\circ(\rT,\cP,\ru)$ using the isometry just described. 

It is important that $x'$ depends on the choice of a component $\hat{T}$ of $T\setminus \{x\}$. When we need to make this dependence explicit we will write $x'(\hat{T})$. 
\subsection{Routing}\label{sec:routing}

In this subsection we let $(C,d,\rho)$ be a rooted $\R$-tree. As this notation may suggest, we will apply the following constructions to a cut-tree; however, the quantities make sense more generally. 
While reading the following definitions, the reader may find it helpful to refer to the description of reconstruction in the discrete setting given in Section~\ref{subsec:general} for intuition. 

Given $v,w \in C$, a {\em signpost} for $v$ and $w$ is a pair $(v_1,w_1)$ with $v_1 \in C_{v \wedge w}^v$ and $w_1 \in C_{v \wedge w}^w$. 
Let $B = \{\emptyset\} \cup \bigcup_{n \in\N} \{0,1\}^n$, and view $B$ as indexing the vertices of a complete infinite binary tree with root $\emptyset$. 
If $b \in \{0,1\}^n$ write $|b|=n$. For $b,b' \in B$, if $b$ is a prefix of $b'$ we write $b \le b'$. This agrees with the genealogical order when $B$ is viewed as a tree. 

A {\em routing} for $v$ and $w$ is a collection $(r_b, b \in B \setminus \{\emptyset\})$ with the following properties. First, $r_0=v$ and $r_1=w$. 
Next, for all $n \ge 1$, for all $b \in \{0,1\}^n$ we have $r_{b0}=r_b$; 
and for all $b \in \{0,1\}^{n-1}$, $(r_{b01},r_{b11})$ is a signpost for $(r_{b0},r_{b1})$.

\begin{figure}
\centering
		\includegraphics[width=0.8\textwidth]{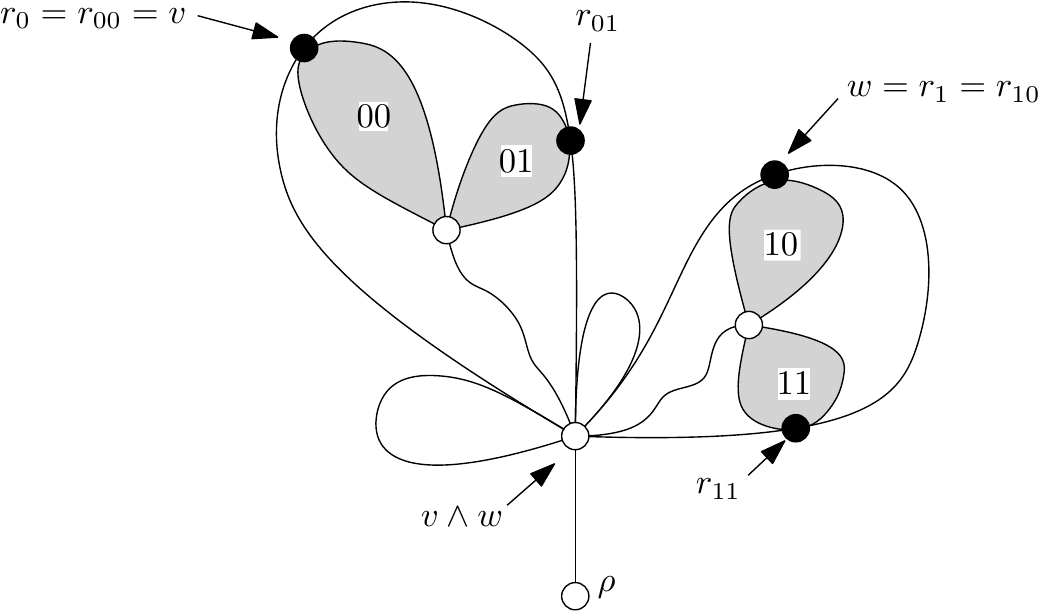}
		\caption{Part of a routing for a pair $v,w$ of points of $C$. The pair $(r_{01},r_{11})$ is a signpost for $(v,w)$. For each $b \in \{0,1\}^2$, we must have $r_{b0}=r_b$, and $r_{b1}$ must lie within the grey subtree labelled with $b$.}
	\label{fig:dec}
\end{figure}

Now fix a set $N \subset C$ such that $\overline{\bigcup_{i \in N} \seg{\rho}{i}}=C$. For each pair $(i,j)$ of distinct elements of $N$ let $r^{ij} = (r^{ij}_b, b 
\in B \setminus \{\emptyset\})$ be a routing for $i$ and $j$. We say the collection $(r^{ij},i,j \in N,i \ne j)$ of routings is \emph{consistent} if the following three properties hold. In words, the first says that $r^{ji}$ is always obtained from $r^{ij}$ by swapping the subtrees at the root of $B$. The second says that sub-routings are themselves routings for the appropriate pairs.  The third says that for any branchpoint $x$ of $C$ and any $y,z \in C_x$, if the subtrees $C_x^y$ and $C_x^z$ are the same then all signposts at $x$ in directions $y$ and $z$ are also the same. 
\begin{enumerate}
\item For any distinct $i,j \in N$, for all $b \in B$, $r^{ij}_{0b} = r^{ji}_{1b}$. 
\item For any distinct $i,j \in N$, distinct $k,l \in N$ and $b \in B$, if $(k,l) = (r^{ij}_{b0},r^{ij}_{b1})$ then 
$r^{kl}_a = r^{ij}_{ba}$ for all $a \in B \setminus \{\emptyset\}$. 
\item 
For any distinct $i,j \in N$ and distinct $k,l \in N$, if $i \wedge j=k\wedge l$ and $C_{i \wedge j}^i=C_{k \wedge l}^k$ 
then $r^{kl}_{01}=r^{ij}_{01}$. 
\end{enumerate}
Suppose $(r^{ij},i,j \in N,i \ne j)$ is consistent. Then, for each $x \in \br(C)$ and each subtree $C'=C_x^y$ above $x$, fix $i,j \in N$ with $i \wedge j=x$ and $C_{i\wedge j}^i=C'$, and let $f(x,C')=r^{i,j}_{01}$. The consistency conditions guarantee that the value of $f(x,C')$ does not depend on $i$ and $j$ satisfying these properties. 

Conversely, suppose that we are given the data $f(x,C')$ for all such pairs $(x,C')$. Then we may reconstruct the collection of routings by setting $r^{i,j}_{01}=f(x,C')$ for all pairs $i,j \in N$ with $i \wedge j=x$ and $C_{i\wedge j}^i=C'$; the consistency conditions then uniquely determine all other routing data. 

\subsection{Routing in cut-trees}\label{sec:rout_cut}
Now suppose that $(C,d,\rho)$ is constructed as in Section~\ref{sec:image}, 
so $(C,d,\rho)$ is the completion of $\cC^\circ=\cC^\circ(\rT,\cP,\ru)$ for a suitable triple $(\rT,\cP,\ru)$ with the property that $\ell(u_i) < \infty$ a.s.\ for all $i \in \N$. Recall that we are treating $\N$ as a collection of random points of $C$. We again assume that $\N$ is dense in $C$.  
For any distinct $i,j \in \N \subset C$, 
the cut-tree construction described above then yields a routing 
$\rR^{ij} = (R^{ij}_b,b \in B \setminus \{\emptyset\})$, built as follows. In reading the description, Figures~\ref{fig:routing} and~\ref{fig:routing2} should be useful. In both figures, the superscripts $ij$ are ommitted for readability.

\begin{figure}
\centering
		\includegraphics[width=0.7\textwidth]{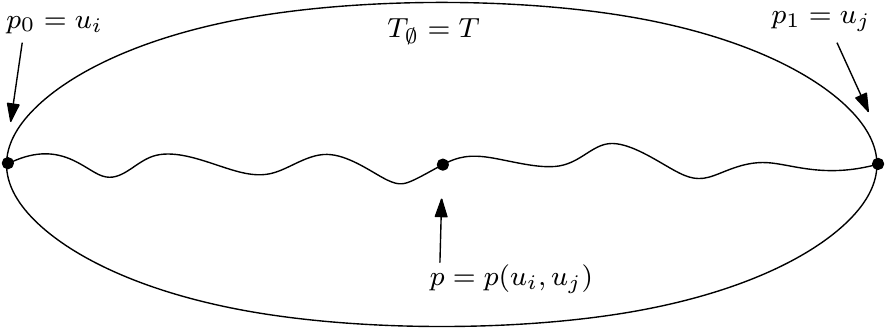}
		\includegraphics[width=0.7\textwidth,page=2]{routings.pdf}
		\includegraphics[width=0.7\textwidth,page=3]{routings.pdf}
	\caption{Top: The tree $T$ with points $u_i$ and $u_j$ marked. 
Middle: the first cut point $p$ and the trees $T_0$ and $T_1$. Bottom: $x$ and $y$ are the first cuts to separate $u_i$ and $u_j$ from $p$, respectively.}
		\label{fig:routing}
\end{figure}
\begin{figure}
\centering
		\includegraphics[width=0.4\textwidth]{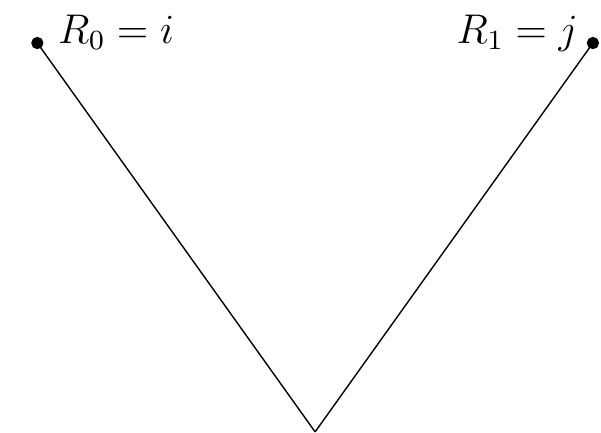}
		\includegraphics[width=0.4\textwidth,page=2]{routing2b.pdf}
		\\
		\vspace{0.5cm}
		\includegraphics[width=0.6\textwidth]{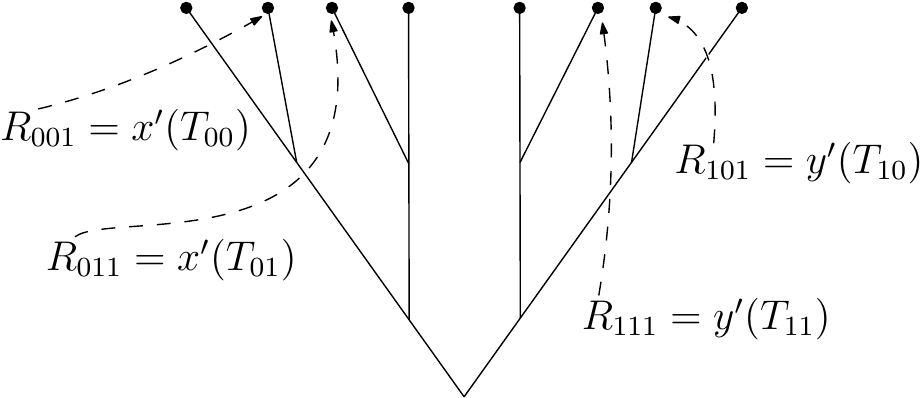}
	\caption{Top left: The images in $C$ of $u_i$ and $u_j$, and the path between them. Top right: The images $R_{01}$ and $R_{11}$ of $p=p_{01}=p_{11}$. Bottom: The images of $x=p_{001}=p_{011}$ and of $y=p_{111}=p_{101}$.}
		\label{fig:routing2}
\end{figure}

First, let $R^{ij}_0=i$ and $R^{ij}_1=j$. Also let $p^{ij}_0=u_i$ and $p^{ij}_1=u_j$. It will later be convenient to set $T^{ij}_{\emptyset}=T$.

Suppose inductively that $(R^{ij}_b,0 < |b| \le n)$ and $(p^{ij}_b,0 < |b| \le n)$ are already defined. Fix $b \in \{0,1\}^{n-1}$ and let $(t,p)$ be the first Poisson point strictly 
separating $p^{ij}_{b0}$ and $p^{ij}_{b1}$. In earlier notation, we have $(t,p) = (t(R^{ij}_{b0},R^{ij}_{b1}),p(R^{ij}_{b0},R^{ij}_{b1}))$. 

Let $T^{ij}_{b0}$ and $T^{ij}_{b1}$ be the components of $T\setminus \{p_k: (t_k,p_k) \in \cP, t_k \le t\}$ containing $p^{ij}_{b0}$ and $p^{ij}_{b1}$, respectively. In earlier notation,  
$T^{ij}_{b0}=T(p^{ij}_{b0},t)$ and $T^{ij}_{b1}=T(p^{ij}_{b1},t)$. 

Recall that $p'(T^{ij}_{b0})$ and $p'(T^{ij}_{b1})$ are the images of $p$ in $C$ corresponding to components $T^{ij}_{b0}$ and $T^{ij}_{b1}$. 
Then set $R^{ij}_{b00}=R^{ij}_{b0}$ and $R^{ij}_{b01}=p'(T^{ij}_{b0})$, and set 
$R^{ij}_{b10}=R^{ij}_{b1}$ and $R^{ij}_{b11}=p'(T^{ij}_{b1})$. 
Finally, set $p^{ij}_{b00}=p^{ij}_{b0}$, $p^{ij}_{b10}=p^{ij}_{b1}$ and $p^{ij}_{b01}=p=p^{ij}_{b11}$. 

Observe that with the above definitions, for all $b \in B\setminus \{\emptyset\}$, $R^{ij}_b$ is an image of $p^{ij}_b$ in $C$. However, it need not be the unique such image, and indeed it is typically not. It is worth recording that the points $(p^{ij}_b,|b|=n)$ are all elements of $\seg{u_i}{u_j}$ (with repetition). It follows from our recursive labelling convention that 
\begin{equation}\label{eq:dij_sum}
d_T(u_i,u_j) = \sum_{|b|=n} d_T(p^{ij}_{b0},p^{ij}_{b1})\, ,
\end{equation}
which will be useful in the next section.

We write $\rR = (\rR^{ij},i,j \in \N,i \ne j)$ for the collection of such routings, or $\rR(\rT,\cP,\ru)$ when we need such dependence to be explicit. 
\begin{prop}\label{prop:consistency}
For all distinct $i,j \in \N$,  $\rR^{ij}$ is a routing 
for $i$ and $j$. Furthermore, $\rR$ is a consistent collection of routings. 
\end{prop}
\begin{proof}
The first statement is by construction. Consistency is immediate from the fact that, in the notation just preceding the proposition, for any $i',j',b'$ with $R^{i'j'}_{b'0}=R^{ij}_{b0}$ and $R^{i'j'}_{b'1}=R^{ij}_{b1}$ we will have $R^{i'j'}_{b'00}=R^{ij}_{b0}$, $R^{i'j'}_{b'10}=R^{ij}_{b1}$, $R^{i'j'}_{b'01}=p'(T^{ij}_{b0})$ and $R^{i'j'}_{b'11}=p'(T^{ij}_{b1})$. 
\end{proof}

We view the triple $(\cC,\N,\rR)$ as the image of $(\rT,\cP,\ru)$ under the cut-tree transformation. It should be understood as a random metric measure space with a countable infinity of marked points, in the sense discussed at the end of Section~\ref{sec:rtrees}. 
The marks are the points $\N$ together with the points $(R^{ij}_b,b \in B\setminus \{\emptyset\}, i,j \in \N,i \ne j)$. Since this is a countable collection, we may re-index it by the natural numbers according to some arbitrary but fixed rule. 
We will describe a convenient such rule in the course of proving Proposition~\ref{prop:jointlaw}. 

\section{The case of stable trees}\label{sec:stable_trees}
Throughout Section~\ref{sec:stable_trees}, we fix $\alpha \in (1,2]$ and let $\rT=(T,d_T,\mu)$ be an $\alpha$-stable tree. We further let $\rU=(U_i,i \in \N)$ be an i.i.d.\ sequence of samples from $\mu$, and let $\cP$ be a Poisson process on $T\times [0,\infty)$ with intensity measure $\Lambda \otimes \d t$, where $\Lambda$ is as defined in Section~\ref{sec:stable_def}. 
For $s \ge 0$, let $F(s) = (F_1(s), F_2(s), \ldots)$ be the sequence of $\mu$-masses of the connected components of $T \setminus \{p:(p,t) \in \cP, t \le s\}$, listed in decreasing order.  Then, as discussed in the introduction, $(F(s), s \ge 0)$ is a self-similar fragmentation process~\cite{AldPi,Mi_05}.  We observe, in particular, that defining $(F(s), s \ge 0)$ as above for any $\alpha \in (1,2]$, by Lemma 10 of \cite{AldPi} and Lemma 8(iii) of \cite{Mi_05}, for each $s \ge 0$, we have $\sum_{i \ge 1} F_i(s) = 1$ almost surely, so that the fragmentation process conserves mass.

We let $\cC=\cC(\rT,\cP,\rU)$ be the cut-tree, and as usual write $\cC=(C, d,\rho,\nu)$.  The following result, which is due to Bertoin and Miermont~\cite{BerMi} in the case $\alpha=2$ and to Dieuleveut~\cite{Dieuleveut} in the case $\alpha \in (1,2)$, states that the cut-tree of a stable tree is again a stable tree.

\begin{thm}\label{thm:bmd}
Let $U_0 \sim \mu$ be a random point of $T$ independent of the points in $\rU$. 
Then we have $(d_T(U_i,U_j), i,j \in \{0\} \cup \N) \eqdist (d(i,j),i,j \in \{\rho\} \cup \N)$.
\end{thm}
In particular, Theorem~\ref{thm:bmd} implies that the assumptions of Proposition~\ref{prop:build_cut} hold, so $\cC$ is well-defined. The theorem then implies that $\cC$ has the same distribution as $(T, d_T, U_0, \mu)$.

Since $\cC$ is an $\alpha$-stable tree, it is compact. In $\rT$, the points $(U_i,i\in \{0\} \cup \N)$ are i.i.d samples from the mass measure. Since, by Proposition~\ref{prop:empirical_measure}, $\nu$ is the empirical measure of $\N$ in $C$, it follows that $(\cC,\{\rho\} \cup \N)$ is distributed as an $\alpha$-stable tree together with a sequence of i.i.d.\ samples from $\nu$. In particular, $\N$ is dense in $C$.

We further let $\rR=(\rR^{ij},i,j \in \N,i \ne j)$ be the collection of routings described in Section~\ref{sec:rout_cut}. By Proposition~\ref{prop:consistency}, $\rR$ is consistent. 

\subsection{Routings in stable trees}

\begin{figure}
\centering
		\includegraphics[width=0.5\textwidth]{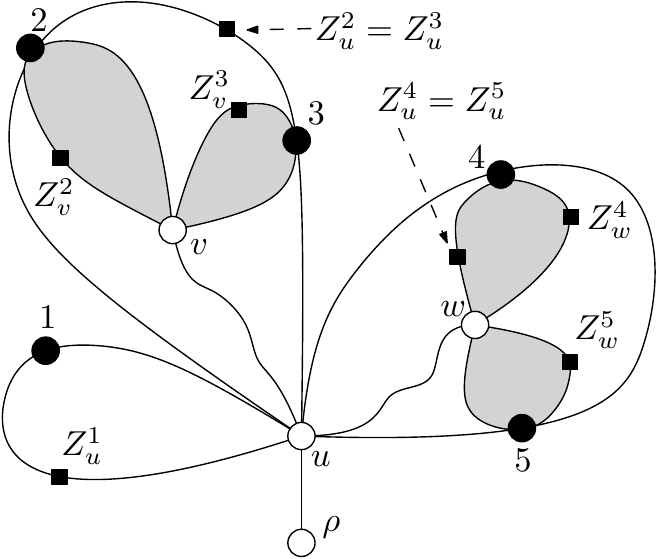}
		\caption{The  routing variables $Z_x^y$ for $x \in \{u,v,w\}$ and $y \in \{1,\ldots,5\}$. The redundancy in the notation is witnessed by the fact that $Z_u^2=Z_u^3$ and $Z_u^4=Z_u^5$.}
	\label{fig:dec2}
\end{figure}

For each $x \in \br(\cC)$ and $y \in C_x$, let $Z_x^y$ be the common value of all the random variables $R^{ij}_{01}$ for which $C^i_{i \wedge j}=C_x^y$, where as before $C_x^y$ is the subtree above $x$ containing $y$. 
We emphasize that there is redundancy in our notation for the set of routing variables $\rZ=\{Z_x^y: x \in \br(\cC),y \in C_x\}$ since there are multiple ways of specifying the same tree $C_x^y$. However, each random variable appears only once; see Figure~\ref{fig:dec2}.
\begin{prop}\label{prop:jointlaw}
The law of $(\cC,\N,\rZ)$ is as follows. 
\begin{enumerate}
\item $\cC$ is a stable tree endowed with its mass measure. 
\item The elements of $\{\rho\} \cup \N$ are i.i.d.\ with law $\nu$.
\item All elements of $\{\rho\} \cup \N$ and $\rZ$ are independent. 
\item For each $x \in \br(\cC)$ and $y \in C_x$, $Z_x^y \sim \nu_{C_x^y}$.
\end{enumerate}
\end{prop}
\begin{proof}
The first two statements are contained in Theorem~\ref{thm:bmd}. 
Next, for $S \subset \N$ write $Z(S)=\{Z_{i \wedge j}^i: i,j \in S, i \ne j\}$ 
To prove the third and fourth statements, it suffices to verify that triple $(\cC,\N,Z(S)))$ has the appropriate law for all finite subsets $S$ of $\N$. In doing so we may assume without loss of generality that $S=[k]$ for some $k \ge 1$. 

We argue by induction on $k$, but before stating our induction hypothesis it is useful to introduce a small amount of notation. For $k \ge 2$, list the elements of $Z([k])$ without repetition as $Z_{x(1)}^{y(1)},\ldots,Z_{x(m)}^{y(m)}$. (So, for example, in Figure~\ref{fig:dec2}, with $k=5$ we have $m=7$ and we may take $(x(1),y(1))=(u,1)$, $(x(2),y(2))=(u,2)$, $(x(3),y(3))=(u,4)$, $(x(4),y(4))=(v,2)$, $(x(5),y(5))=(v,3)$, $(x(6),y(6))=(w,4)$ and $(x(7),y(7))=(w,5)$.) 
We omit the dependence on $k$ from our notation. 

We will show by induction that given $(\cC,\{\rho\} \cup \N)$, for all $1 \le i \le m$, $Z_{x(i)}^{y(i)}$ has law $\nu_{C_{x(i)}^{y(i)}}$ and, moreover, the random variables $Z_{x(i)}^{y(i)}$ are independent. This identifies the law of the $Z_x^y$. It also shows that the only dependence between the $Z_x^y$ and $\{\rho\} \cup \N$ is via the labelling of subtrees, and thus establishes (3) and (4). 

The case $k=1$ is immediate as $Z([1])$ is empty, and the joint law of the points in $\{\rho\} \cup \N$ is given by Theorem~\ref{thm:bmd}. We hereafter assume $k \ge 2$. 

Let $b$ be the nearest branchpoint to $\rho$ in $\bigcup_{i \in [k]} \seg{\rho}{i}$. List the elements of $\{\rC_b^i,i \in [k]\}$ without repetition as $\rC_b(1),\ldots,\rC_b(\ell)$, with $\rC_b(j)=(C_b(j),d,b)$. Note that $2 \le \ell \le k$, so that for each $1 \le j \le \ell$ the subtree $\rC_b(j)$ contains $k_j \le k-1$ elements of $[k]$. We list the points of $\N$ lying in $C_b(j)$  in increasing order as $\rN_j=(n(j,m),m \ge 1)$. Observe that the sets $\rN_j \cap [k]$, with $1 \le j \le \ell$, partition $[k]$. 

There is a unique point $(p,t)$ of $\cP$ which first strictly separates $U_{n(1,1)}$ from $U_{n(2,1)}$ in $T$. In the notation of Section~\ref{sec:cut_tree_def}, this is the point $(p,t)=(p(U_{n(1,1)},U_{n(2,1)}),t(U_{n(1,1)},U_{n(2,1)}))$. Note that this is also the first separator of any $U_{n(i,1)}$ and $U_{n(j,1)}$ for distinct $i$ and $j$. 
 
Recall that $T(U_{n(i,1)},t)$ is the connected component of $T\setminus \{p_j:t_j \le t\}$ containing $U_{n(i,1)}$. Since this does not contain the point $p$ we let 
$T^*(U_{n(i,1)}) = T(U_{n(i,1)},t) \cup \{p\}$ and $\rT^*(U_{n(i,1)})$ be the subtree of $\rT$ induced by $T^*(U_{n(i,1)})$. 

In~\cite{Mi_05}, Miermont showed that $\rT^*(U_{n(i,1)})$ is a stable tree of mass $\mu(T^*(U_{n(i,1)}))$; see, in particular, the proof of his Lemma 9. Furthermore, writing $\cP^*(U_{n(i,1)}) = \{(p_j,t_j-t): p_j \in T(U_{n(i,1)},t)\}$, then $\cP^*(U_{n(i,1)})$ is a Poisson process with intensity measure $\lambda|_{T^*(U_{n(i,1)})}\otimes \d t$. The pairs $((\rT^*(U_{n(i,1)}),\cP^*(U_{n(i,1)})),i \in [\ell])$ only depend on each other via the vector of masses $(\mu(T^*(U_{n(i,1)})),i \in [\ell])$. Finally, for each $i \in [\ell]$, the points $\rU_i = (U_j: j \in \rN_i)$ are precisely those $U_j$ lying in $T^*(U_{n(i,1)}))$, and are i.i.d.\ with law $\mu_{T^*(U_{n(i,1)})}$. By Theorem~\ref{thm:re-rootinvariance}, $p$ also has law $\mu_{T^*(U_{n(i,1)})}$.

For each $1 \le i \le \ell$, the tree $\rC_b(i)$, which is rooted at $b$, is obtained as the cut-tree of $(\rT^*(U_{n(i,1)}),\cP^*(U_{n(i,1)}),\rU_i)$. It follows that $\rC_b(i)$ is a stable tree of mass $\nu(C_b(i))=\mu(T^*(U_{n(i,1)}))$. The facts from the preceding paragraph also imply that the trees $(\rC_b^i,i \in [\ell])$ are conditionally independent given their masses.
Furthermore, the elements of $\rN_i$ are precisely the images in $C_b(i)$ of the points in $\rU_i$, and $Z_b^{n(i,1)}$ is the image of $p$. 
These facts are special cases of the observation about cut-trees of subtrees described at the end of Section~\ref{sec:image}. Finally, note that we may also view $\rC_b(i)$ as the cut-tree of 
$(\rT^*(U_{n(i,1)}),\cP^*(U_{n(i,1)}),\{p\} \cup \rU_i)$, as described near the start of Section~\ref{sec:image}. 

Now apply the inductive hypothesis to 
\[
\left(\rC_b(i),\left(b,Z_b^{n(i,1)},n(i,1),n(i,2),\ldots\right),Z(\rN_i \cap [k]) \setminus \{Z_b^{n(i,1)}\} \right)\, ,
\] 
for each $i \in [\ell]$. This is permitted since $b$, $Z_b^{n(i,1)}$, and the elements of $\rN_i$ all have the correct laws, and since $|N_i \cap [k]|=k_i < k$. 

We obtain by induction that $Z_b^{n(1,1)},\ldots,Z_b^{n(\ell,1)}$ are independent, with $Z_b^{n(i,1)} \sim \nu_{C_b(i)}$. We emphasize that this is because $Z_b^{n(i,1)}$ is the image of a uniform point in $T^*(U_{n(i,1)})$; in the induction $Z_b^{n(i,1)}$ is no longer playing the role of a routing variable. 

We further obtain that for each $1 \le i \le \ell$, the elements of $Z(\rN_i \cap [k]) \setminus \{Z_b^{(n(i,1))}\}$ have the correct joint law and are independent of $Z_b^{n(i,1)}$. Here we are again using the difference in roles between $Z_b^{n(i,1)}$ and the other elements of $Z(\rN_i \cap [k])$. Finally, the subtrees $C_b(i)$ are conditionally independent given their masses, which yields the requisite independence of the collections $Z(\rN_i \cap [k])$ for $1 \le i \le \ell$. 
Since 
\[
Z([k]) = \bigcup_{i=1}^{\ell} Z(\rN_i \cap [k])\, ,
\]
this fully identifies the joint law of the random variables $Z_{x(1)}^{y(1)},\ldots,Z_{x(m)}^{y(m)}$ which comprise $Z([k])$, and so completes the proof. 
\end{proof}
Note that since $\rR$ is consistent, it is completely determined by $\rZ$. The preceding proposition therefore fully specifies the joint law of the triple $(\cC,\N,\rR)$. 

\subsection{Distributional identities for stable trees}
\label{sec:dist_ident}
We recall the definitions of some distributions that play a role in the sequel. Write $\Delta_n=\{(x_1,\ldots,x_n) \in \R_+^n:\sum_{i=1}^n x_i=1\}$. A $\Delta_n$-valued random vector $\bX=(X_1,\ldots,X_n)$ has the {\em Dirichlet distribution} $\dir(\theta_1,\ldots,\theta_n)$ if its density with respect to Lebesgue measure on $\Delta_n$ is 
\[
\frac{\Gamma(\sum_{i=1}^n \theta_i)}{\prod_{i=1}^n \Gamma(\theta_i)} \prod_{j=1}^n x_i^{\theta_i-1}. 
\]

A non-negative random variable $Y$ has the {\em Mittag-Leffler distribution with parameter $\beta \in (0,1)$}, denoted $\ml(\beta)$, if it satisfies 
\begin{equation}\label{eq:ml}
\E{Y^p} = \frac{\Gamma(p+1)}{\Gamma(p\beta+1)}\, ,
\end{equation}
for $p \ge -1$. This equation determines the law of $Y$ (see \citep[p.12]{pitman06} and \citep[p.391]{bingham}).
Write $g_{\beta}$ for the density of $\ml(\beta)$ with respect to Lebesgue measure. 
Write $\widehat{\ml}(\beta)$ for the size-biased distribution, which has density 
\[
\hat{g}_{\beta}(r) := \Gamma(\beta+1) r g_{\beta}(r), \quad r \ge 0.
\]
\medskip

Our proofs exploit the following characterization of the size-biased Mittag-Leffler distribution, 
\begin{lem} \label{T:ML_characterization} 
Fix $\beta \in [1/2,1)$, and let $M$, $M_1$ and $M_2$ be independent and identically distributed non-negative random variables with 
\[
\E{M}=\frac{2\Gamma(\beta+1)}{\Gamma(2\beta+1)}\, .
\]
Let $(X_1, X_2, X_3) \sim \dir(\beta,\beta,1-\beta)$ be independent of $M_1$ and $M_2$. 
Then $M$ solves the recursive distributional equation
\begin{gather} \label{E:RDE}
X_1^{\beta} M_1 + X_2^{\beta} M_2 \eqd M,
\end{gather}
if and only if $M \sim \widehat{\ml}(\beta)$. 
\end{lem}

Let $x,y,z$ be independent points of $T$ with common law $\mu$, and let $b$ be the common branchpoint of $x,y,z$ (i.e. the unique element of $\seg{x}{y} \cap \seg{x}{z} \cap \seg{y}{z}$). Recall that $T^x_b$ is the subtree of $T$ consisting of all points $w$ with $b \not\in \oseg{x}{w}$.

\begin{thm}\label{T:cut_in_3b}
\begin{enumerate}
\item The random variable $\alpha \cdot d(x,y)$ is $\widehat{\ml}(1-\frac{1}{\alpha})$-distributed. 
\item The vector $(\mu(T_b^x),\mu(T_b^y),\mu(T \setminus (T_b^x \cup T_b^y)))$ is $\dir(1-\frac{1}{\alpha},1-\frac{1}{\alpha},\frac{1}{\alpha})$-distributed. 
\end{enumerate}
\end{thm}

\begin{proof}
For the first assertion, see Theorem 3.3.3 of \cite{DuqLG_02}. The second follows from Corollary 10 of \cite{HaasPitmanWinkel} after using distributional identities established in \cite{GoldHa}, and is explicitly noted as \citep[(4.1)]{GoldHa}.  When $\alpha=2$, the law of the distance between two uniform points was earlier proved to follow the Rayleigh distribution, by Aldous \cite{CRT3}. This is in agreement with the current result, up to a choice of normalization, since if $Z$ is standard Rayleigh then the law of $\sqrt{2}Z$ is $\widehat{\ml}(1/2)$; see Section 1.1 of \cite{GoldHa}.
\end{proof}

We can now proceed to the proof of Lemma~\ref{T:ML_characterization}.

\begin{proof}[Proof of Lemma~\ref{T:ML_characterization}] Write $\beta=1-1/\alpha$; then $\alpha \in (1,2]$. 
Consider again the $\alpha$-stable tree $\rT$ with $x,y,z$ independent points sampled according to the law $\mu$.  Let $M = \alpha \cdot d(x,y)$, $M_1 = \alpha \mu(T_b^x)^{-\beta} \cdot d(x,b)$, $M_2 = \alpha \mu(T_b^y)^{-\beta}\cdot d(y,b)$, $X_1 = \mu(T_b^x)$, $X_2 = \mu(T_b^y)$ and $X_3 = \mu(T_b^{z})$.  Then $M = X_1^{\beta} M_1 + X_2^{\beta} M_2$.  Moreover, by Theorems~\ref{T:cut_in_3} and \ref{T:cut_in_3b}, $M$ is $\widehat{\ml}(\beta)$-distributed, and $M_1, M_2$ are independent $\widehat{\ml}(\beta)$ random variables, independent of $(X_1,X_2,X_3) \sim \dir(\beta,\beta,1-\beta)$.  It follows that $\widehat{\ml}(\beta)$ satisfies the RDE (\ref{E:RDE}).

For the converse, we use that the left-hand side of \eqref{E:RDE} is an instance of the \emph{smoothing transform} applied to the law of $M$.  The fixed points of the smoothing transform have been completely characterized by Durrett and Liggett \cite{DurLig}.  Indeed, the space of fixed points is determined by the analytical properties of the function $\nu: \R_+ \to \R$ defined (in our setting) by
\begin{gather*}
\nu(s) = \log \left( \E{X_1^{\beta s} \ind_{X_1 > 0} + X_2^{\beta s} \ind_{X_2 > 0} } \right),
\end{gather*}
for $s \geq 0$.  Since $X_1$ and $X_2$ are marginally both distributed as $\mathrm{Beta}(\beta,1)$, it is easily checked that $X_1^{\beta} \eqd X_2^{\beta} \eqd U$, where $U$ is uniform on $[0,1]$.  Hence,
\begin{gather*}
\nu(s) = \log\left(2 \E{U^s}\right) = \log(2) - \log(s + 1), \quad s \geq 0,
\end{gather*}
for any $\alpha \in (1,2]$.  Observe that $\nu$ has its unique zero in $(0,1]$ at $s = 1$, and that $\nu'(1) = -1/2 < 0$.  In this case, \cite[Theorem 2(a)]{DurLig} entails that \eqref{E:RDE} has a unique distributional solution, up to multiplication by a non-negative constant; we have already identified this solution as the $\widehat{\ml}(\beta)$ distribution. 
\end{proof}

\subsection{Reconstructing the distance between a pair of points in $\rT$}
\label{sec:reconstruct}

Recall that $\rC=(C,d,\rho,\nu)$ is the cut-tree of $\rT$, that the points of $\N$ are i.i.d.\ with law $\nu$ and that $\rR$ is the collection of routings for the elements of $\N$. 
For the remainder of Section~\ref{sec:reconstruct} we fix distinct $i,j \in \N$ and recall that the routing for $i$ and $j$ is denoted $\rR^{ij}$.

Recall that for distinct nodes $y$ and $z$ of $C$, the subtree of $C$ above $z$ containing $y$ is denoted by $C^y_z$. Now fix $b \in B\setminus \{\emptyset\}$, and let $b'$ be the sibling of $b$ in $B$; so if $b=\hat{b}0$ then $b'=\hat{b}1$ and vice versa. 
Then let $M_b = \nu(C^{R_{b}}_{R_b \wedge R_{b'}})$. 
In Figure~\ref{fig:dec}, for example, $M_{00}$ is the mass of the shaded subtree labelled $00$, and likewise for $M_{01}$, $M_{10}$ and $M_{11}$.
It is crucial in what follows that, in the notation of Section~\ref{sec:rout_cut} we also have $M_b=\mu(T^{ij}_b)$ since $C^{R_{b}}_{R_b \wedge R_{b'}}$ is the cut-tree of $T^{ij}_b$. 
We also set $M_{\emptyset}=\nu(C)=1=\mu(T)$. 

\medskip
It is convenient to write 
\[
(\Delta_{b0},\Delta_{b1},1-\Delta_{b0}-\Delta_{b1}) = \frac{1}{M_b} \pran{M_{b0},M_{b1},M_b-M_{b0}-M_{b1}}. 
\]
We will repeatedly use that for all $b \in B$, 
\[
(\Delta_{b0},\Delta_{b1},1-\Delta_{b0}-\Delta_{b1}) \sim \dir\bigg(1-\frac{1}{\alpha},1-\frac{1}{\alpha},\frac{1}{\alpha}\bigg)\, ,
\]
and that the vectors $\{(\Delta_{b0},\Delta_{b1},1-\Delta_{b0}-\Delta_{b1}),b \in B\}$ are mutually independent; 
these properties follow from Theorems~\ref{T:cut_in_3} and \ref{T:cut_in_3b}.

Next, for $n \ge 0$, let 
\begin{equation}\label{eq:zzprime}
Y_n=Y_n(i,j) = 
\sum_{|b|=n} M_b^{1-1/\alpha}\, .
\end{equation}
\begin{prop}\label{prop:znlimit}
As $n \to \infty$, $Y_n \convas Y$ where the limit $Y = Y(i,j)$ satisfies 
\[
\frac{2\Gamma(2-\frac{1}{\alpha})}{\Gamma(3-\frac{2}{\alpha})}Y\sim\widehat{\ml}(1-1/\alpha).
\] 
\end{prop}
\begin{proof}
Let $\cG_n = \sigma((M_b,|b| \le n))$, and let $\cG_{\infty} = \sigma((M_b, b \in B)) = \sigma(\bigcup_{n} \cG_n)$. 

Explicit calculation (as in the proof of Lemma~\ref{T:ML_characterization}) shows that $\Delta_{b0}^{1-1/\alpha}$ and $\Delta_{b1}^{1-1/\alpha}$ are both $\mathrm{U}[0,1]$-distributed, 
so 
\[
\Cexp{Y_{n+1}}{\cG_n} = 
\sum_{|b|=n}M_b^{1-1/\alpha} \cdot \Cexp{\Delta_{b0}^{1-1/\alpha}+\Delta_{b1}^{1-1/\alpha}}{\cG_n} = Y_n
\]
and so $(Y_n)$ is a $(\cG_n)$-martingale. 
Since $Y_n \ge 0$ for all $n$, it follows that $Y_n \convas Y$ for some random variable $Y$ by the martingale convergence theorem; it remains to show that $Y$ has the correct distribution. For this we use a second martingale argument together with a result of \cite{AB}.

Fix $n \in \N$ and let $(Z^{(n)}_b: |b|=n)$ be i.i.d.\ $\widehat{\ml}(1-\frac{1}{\alpha})$. Then for $b \in B$ with $|b|=m<n$, define 
$Z^{(n)}_b$ inductively by 
\[
Z^{(n)}_b = \Delta_{b0}^{1-1/\alpha} Z^{(n)}_{b0}+\Delta_{b1}^{1-1/\alpha} Z^{(n)}_{b1}\, .
\]
By Lemma~\ref{T:ML_characterization}, $Z^{(n)}_b\sim \widehat{\ml}(1-\frac{1}{\alpha})$ for all $b$ with $|b| \le n$. Furthermore, the families $(Z^{(n)}_b,|b| \le n)$ are consistent in $n$, in that 
\[
(Z^{(n)}_b,|b| \le n-1) \sim (Z^{(n-1)}_b,|b| \le n-1), 
\]
and so have a projective limit by Kolmogorov's extension theorem. Let $(Z_b,b \in B)$ be such that $(Z_b,|b| \le n) \sim (Z^{(n)}_b,|b| \le n)$ for all $n$; in particular, for all $b$ we have $Z_b \sim \widehat{\ml}(1-\frac{1}{\alpha})$, and 
\begin{equation}\label{eq:zu_recursion}
Z_b = \Delta_{b0}^{1-1/\alpha} Z_{b0}+\Delta_{b1}^{1-1/\alpha} Z_{b1}\, .
\end{equation}
The families $(Z_b, b \in B)$ and $(\Delta_b,b \in B)$ together define a {\em recursive tree process} in the sense of \cite{AB}. 
This process is easily seen to verify the conditions of Corollary 17 of \cite{AB} (briefly: $\E{\Delta_{b,1}^x}$ is decreasing in $x$ and $\p{Z_{\emptyset}=0}=0$). We conclude that the recursive tree process is {\em endogenous}, which means that for all $b \in B$, 
the random variable $Z_b$ is measurable with respect to $\sigma(\Delta_{b'}, b < b')$. In particular, $Z_{\emptyset}$ is integrable and $\cG_{\infty}$-measurable, and so the martingale convergence theorem gives that
\[
\Cexp{Z_{\emptyset}}{\cG_n} \convas Z_{\emptyset}\, 
\]
as $n \to \infty$. 

Finally, by (\ref{eq:zu_recursion}) and induction we have $Z_{\emptyset} = \sum_{|b|=n} M_b^{1-1/\alpha} Z_b$ for all $n$. 
Also, $Z_b$ is $\sigma(\Delta_{b'}, b < b')$-measurable, so if $|b|=n$ then $Z_b$ is independent of $\cG_n$. On the other hand, $M_b$ is $\cG_n$-measurable and so 
\[
\Cexp{Z_{\emptyset}}{\cG_n} = \sum_{|b|=n} M_b^{1-1/\alpha} \E{Z_b} = \frac{2\Gamma(2-\frac{1}{\alpha})}{\Gamma(3-\frac{2}{\alpha})}\cdot Y_n, 
\]
where the last equality holds as $Z_b \sim \widehat{\ml}(1-\frac{1}{\alpha})$. 
It follows that $\frac{2\Gamma(2-\frac{1}{\alpha})}{\Gamma(3-\frac{2}{\alpha})}\cdot Y=Z_{\emptyset}$ almost surely.
\end{proof}

Let
\begin{equation}\label{eq:reconst_dist}
\delta_{C}(i,j) = \frac{2 \alpha \Gamma(2 - \frac{1}{\alpha})}{\Gamma(3 - \frac{2}{\alpha})} \cdot Y(i,j).
\end{equation}
Observe that this has the same law as $d_T(U_i,U_j)$.  In Theorem~\ref{thm:dist_point_recon} below, we will show that the two quantities are, in fact, almost surely equal.

Next, for $b \in B$ write 
\[
Y_{n}^b = Y_n^b(i,j)= \sum_{|b'|=n, b \le b'} M_{b'}^{1-1/\alpha}.
\]
A practically identical proof then shows that for all $b \in B$, $Y_{n}^b \convas Y^b=Y^b(i,j)$, for random variables $(Y^b,b \in B)$ all a.s.\ satisfying $Y^b = Y^{b0}+Y^{b1}$. In this notation we have $Y(i,j) = Y^{\emptyset}(i,j)$. In particular, this allows us to define 
\begin{equation} \label{eq:reconst_dist2}
\delta_C(i,i \wedge j) = \frac{2 \alpha \Gamma(2 - \frac{1}{\alpha})}{\Gamma(3 - \frac{2}{\alpha})} \cdot Y^0(i,j), 
\quad \text{and}\quad 
\delta_C(j,i \wedge j) = \frac{2 \alpha \Gamma(2 - \frac{1}{\alpha})}{\Gamma(3 - \frac{2}{\alpha})} \cdot Y^1(i,j).  
\end{equation}
The relation between the $Y^b$ then implies that $\delta_C(i,j) = \delta_C(i,i \wedge j)+\delta_C(j,i \wedge j)$

\subsection{Recovering $(\rT,\cP,\rU)$ from $(\cC,\N,\rR)$}\label{sec:recover_all_dist}
Recall that a.s.\ for all distinct $i,j\in \N$, $\rR^{ij}$ is a routing for $i$ and $j$ by Proposition~\ref{prop:consistency}. 
For distinct $i,j \in \N$, let $\delta_{C}(i,j)$, $\delta_C(i, i \wedge j)$ and $\delta_C(j, i \wedge j)$ be defined as in (\ref{eq:reconst_dist}) and (\ref{eq:reconst_dist2}). Note that $i$ and $j$ have law $\nu$, so by Proposition~\ref{prop:znlimit}, $Y(i,j)$ is well-defined and $Y(i,j)\cdot 2 \Gamma(2 - \frac{1}{\alpha})/\Gamma(3 - \frac{2}{\alpha})$ is $\widehat{\ml}(1-1/\alpha)$-distributed.

Let $\pi(i,j)$ be the unique element of $\seg{U_i}{U_j}$ at distance $\delta_{C}(i,i \wedge j)$ from $U_i$. Note that $\pi(i,j)$ is an element of $T$, not of $C$. 
Finally, let 
\[
\tau(i,j) = \int_{\seg{\rho}{i\wedge j}} \frac{1}{\nu(C_z)} \mathrm{d}z\, ,
\]
where the integral is with respect to the length measure on $\seg{\rho}{i\wedge j}$. 
\begin{thm}\label{thm:dist_point_recon}
The following all hold almost surely.  
\begin{enumerate}
\item $d_T(U_i,p(U_i,U_j))=\delta_{C}(i,i \wedge j)$ and 
$d_T(U_j,p(U_i,U_j))=\delta_{C}(j,i \wedge j)$, and thus $d_T(U_i,U_j)=\delta_{C}(i,j)$. 
\item $(t(U_i,U_j),p(U_i,U_j)) = (\tau(i,j),\pi(i,j))$.
\end{enumerate}
\end{thm}
\begin{proof}
For each $b \in B$, let 
\[
D_b = \frac{d_T(p^{ij}_{b0},p^{ij}_{b1})}{M_b^{1-1/\alpha}}\, .
\]
The reader may wish to glance at Figure~\ref{fig:routing} to refresh the definitions of the points $p^{ij}_{b0}$ and $p^{ij}_{b1}$, and their relation to $T^{ij}_b$. (When consulting that figure, it may be useful to take $b=11$, say, for concreteness. Also recall that the superscripts $ij$ are ommitted from the figure for legibility.)

For each $n \ge 0$, the trees $(T^{ij}_b,|b|=n)$ are rescaled $\alpha$-stable trees and are conditionally independent given their masses. Moreover, the random variables $(\alpha^{-1}D_b,|b|=n)$ are i.i.d.\ and are $\widehat{\ml}(1-1/\alpha)$- distributed. These observations are consequences of Theorems~\ref{T:cut_in_3} and \ref{T:cut_in_3b}. 

Now note that, by (\ref{eq:dij_sum}), we have 
\[
\sum_{|b|=n} D_b M_b^{1-1/\alpha} = \sum_{|b|=n} d_T(p^{ij}_{b0},p^{ij}_{b1}) = d_T(U_i,U_j),
\]
for every $n$, and thus we trivially have 
\[
\lim_{n\to \infty} \sum_{|b|=n} D_b M_b^{1-1/\alpha} = d_T(U_i,U_j). 
\]
On the other hand, $(D_b,|b|=n)$ is independent of $(M_b,b \in B)$, so with $\cG_n=\sigma((M_b,|b| \le n))$ as in Proposition~\ref{prop:znlimit}, we have 
\[
\Cexp{d_T(U_i,U_j)}{\cG_n} 
= \E{D_{\emptyset}} \cdot \sum_{|b|=n} M_b^{1-1/\alpha} 
= \E{D_{\emptyset}} \cdot Y_n(i,j). 
\]
Taking $n$ to infinity, it follows that 
\[
\Cexp{d_T(U_i,U_j)}{\cG_\infty} \aseq  \E{D_{\emptyset}} \cdot Y(i,j). 
\]
But $d_T(U_i,U_j) \eqdist \E{D_{\emptyset}} \cdot Y(i,j)$, which implies that, in fact, $d_T(U_i,U_j)$ is $\cG_{\infty}$-measurable (see \cite{durrett}, Exercise 5.1.12). We thus have that $d_T(U_i,U_j)\aseq \E{D_{\emptyset}} \cdot Y(i,j)$.
An essentially identical proof shows that 
$d_T(U_i,p(U_i,U_j))\aseq \E{D_{\emptyset}} \cdot Y(i,i \wedge j)$ and that 
$d_T(U_j,p(U_i,U_j))\aseq \E{D_{\emptyset}} \cdot Y(j,i \wedge j)$. Since 
\[
\E{D_{\emptyset}} = \frac{2 \alpha \Gamma(2 - \frac{1}{\alpha})}{\Gamma(3 - \frac{2}{\alpha})}\, ,
\]
this establishes the first claim of the theorem.

For the second claim, by definition we have 
$\alpha(i \wedge j) = t(U_i,U_j)$, and (\ref{eq:alpha_formula}) 
also gives 
\[
\alpha(i \wedge j) = \inf\left\{t: \int_0^t \mu(T(U_i,r))\mathrm{d}r \ge d(\rho,i \wedge j)\right\}. 
\]
It is convenient to parametrize $\left[ \! \left[ \rho , i \right[ \! \right[$ by length; to this end, for $\gamma \in [0, \ell(U_i))$, write $z(\gamma)$ for the unique point $z \in \left[ \! \left[ \rho , i \right[ \! \right[$ with $d(\rho,z) = \gamma$.  Then for all such $\gamma$ we have
\[
\int_0^{\alpha(z(\gamma))} \mu(T(U_i,r)) \mathrm{d} r = \gamma,
\]
from which it follows that
\[
\alpha(z(\gamma)) = \int_0^{\gamma} \frac{1}{\mu(T(U_i, \alpha(z(y))))} \mathrm{d} y.
\]
Recall that we also have $\alpha(i \wedge j) = t(U_i,U_j)$. 
The result will thus follow if we can show that $\mu(T(U_i,\alpha(z))) = \nu(C_z)$ for $z \in \left[ \! \left[ \rho , i \right[ \! \right[$, 
by taking $\gamma=d(\rho,i \wedge j)$ so that $z(\gamma)=i \wedge j$.  We have
\[
\{j \in \N: j \in C_z\} = \{j \in \N: U_j \in T(U_i, \alpha(z))\}.
\]
We also have
\[
\mu(T(U_i, \alpha(z))) = \lim_{n \to \infty} \frac{1}{n} \#\{j \le n: U_j \in T(U_i, \alpha(z))\}.
\]
by the Glivenko-Cantelli theorem and 
\[
\nu(C_z) = \lim_{n \to \infty} \frac{1}{n} \#\{j \le n: j \in C_z\}
\]
by Proposition~\ref{prop:empirical_measure}. This completes the proof. 
\end{proof}
\begin{cor}\label{cor:inv_meas}
The triple $(\rT,\rU,\cP)$ is measurable with respect to the triple $(\cC,\N,\rR)$. \end{cor}
\begin{proof}
First, since $\rU$ is a.s.\ dense in $\rT$, the collection of pairwise distances $(\delta_{C}(i,j),i,j \in \N) = (d_T(U_i,U_j),i,j \in \N)$ uniquely reconstructs $(T,d_T)$ up to metric space isometry, and further reconstructs the sequence $\rU$ of points of $T$. Next, since $\mu$ is the empirical measure of the collection $\rU$, this also reconstructs $\mu$ and thus reconstructs $\rT=(T,d_T,\mu)$ up to measured metric space isometry. 

Finally note that, almost surely, every point $(t,p) \in \cP$ separates some pair of points from the sequence $\rU$. In other words, every element of $\cP$ may be represented as $(t,p)=(t(U_i,U_j),p(U_i,U_j))$ for some $i,j \in \N$. 
It follows from the second statement of Theorem~\ref{thm:dist_point_recon} that, almost surely, we may reconstruct $\cP$ from $\cC$ and the routings $\rR$ as 
\[
\cP = \{(\tau(i,j),\pi(i,j)): i,j \in \N,i \ne j\}. 
\qedhere
\]
\end{proof}

The proof of Corollary~\ref{cor:inv_meas} describes a specific measurable map, which we now denote $\Phi$, with the property that $\Phi(\rC,\N,\rR) \aseq (\rT,\rU,\cP)$. The map $\Phi$ is built using $\delta_C$, the empirical measure, and the points $(\tau(i,j),\pi(i,j))$.
Denoting the laws of the triples $(\cC,\N,\rR)$ and $(\rT,\rU,\cP)$ by $\mathcal{L}$ and $\mathcal{M}$, respectively, this immediately entails the following corollary. 
\begin{cor}
Let $(\cC',\rN,\rR')$ be any random variable with law $\mathcal{L}$. Then 
$\Phi(\cC',\rN,\rR')$ has law $\mathcal{M}$. 
\end{cor}


\section{Questions and perspectives}
\label{sec:conc}
Though there are now several important cases in which the cut-trees and their reconstructions are well-understood, it remains to develop a fully general theory.
It would be interesting to develop a more general theory of cut-trees and their reconstructions. The following list of questions provide some concrete avenues for research along these lines. 

\begin{enumerate}
\item 
The results on empirical measures in Section~\ref{sec:measures} require compactness of the cut tree. In the case of $\alpha$-stable trees considered in this work, compactness follows from existing results in the literature. More generally, though it is likely possible to prove compactness {\em ad hoc} for specific models, it would be interesting to develop general sufficient conditions for compactness of the cut tree of an $\R$-tree. 
\item Are there cases other than those addressed by the current paper or by Broutin and Wang~\cite{BroutinWangptrees} where the cut tree has the same law as the original tree? 
\item What conditions on the law of $(\rT,\rU,\cP)$ are sufficient to guarantee that the triple may almost surely be reconstructed from $(\rC,\N,\cR)$?
\item 
For a given triple $(\rt,\ru,\cP)$ even if the cut tree $C$ is not compact, it may be that in some cases the images of the points in $\ru$ define an ``empirical measure'' on $C$. When does this occur? 
\item The distributional identities that this paper is about are ``annealed'' in that one averages over the realization of the tree, the sampled points and the cuts. It would also be interesting to study the above properties (compactness, sampled points dense etc) for a fixed tree. 
\end{enumerate}
In the case of $\alpha$-stable trees, there are also interesting unanswered questions; here are two which we find worthy of study, one quite concrete and the other rather vague.
\begin{enumerate}\setcounter{enumi}{5}
\item What is the law of the cut-tree of an $\alpha$-stable tree if the driving Poisson process is has intensity $\lambda \otimes \mathrm{d}t$, where $\lambda$ is the length measure? In other words, what happens if cuts fall uniformly on the skeleton rather than at branch points? 

\item The map $\Phi$ almost surely reconstructs $(\mathrm{T},\mathrm{U},\mathcal{P})$ from $(C,\N,\mathrm{R})$. Is $\Phi$ stable under small perturbations of $(C,\N,\mathrm{R})$? To formalize such a statement, one would need to define a more robust reconstruction map $F$, presumably extending the definition of $\Phi$. 
Having found an appropriate generalization, the question is then whether $F$ has the property that if $(C_k,\N,\mathrm{R}_k) \convas (C,\N,\mathrm{R})$ as $k \to \infty$ then $F(C_k,\N,\mathrm{R}_k) \convas F(C,\N,\mathrm{R})$. 
\end{enumerate}
A stability statement such as the second one would allow one to deduce distributional information about a random tree from information about its cut tree. The next and final section of the paper describes a concrete situation in which this would be useful: a model of discrete random trees with a complicated law, but for which a tree obtained by the discrete version of the reconstruction map has a simple and explicit description.

\subsection{A stationary tree aggregation process}

The discrete process of reconstruction described in Section \ref{subsec:general} arises in a somewhat different setting, which provides an additional motivation for its study.  (This arose in discussions of the third author with Edward Crane, Nic Freeman, James Martin, B\'alint T\'oth and Dominic Yeo.)  We describe a rooted tree-valued process which grows until it becomes infinite, and is then ``burnt'' back to the root.  This is intended to model a mean-field forest fire process (see \cite{RathToth,CraneFreemanToth}), viewed from the perspective of a particular vertex, but the precise details of this interpretation are unnecessary here.

Fix a probability distribution $\cW$ on the set of rooted trees and consider a rooted tree-valued Markov process $(T_0(t), \rho)$, which evolves as follows.  Start from a single vertex $T_0(0) = \rho$, the root.  At any subsequent time $t$, given that the current state is $T_0(t) = (T,\rho)$, at rate given by the number of vertices of $T$, sample a rooted tree $(T', r)$ from $\cW$ and an independent random vertex $v$ from $T$ and attach $r$ to $v$ by an edge, and root the resulting tree at $\rho$.  It is possible for the jump-times $J_1 < J_2 < \dots$ of this process to accumulate (i.e.\ $J_m \to J_{\infty}$ as $m \to \infty$ for some $J_{\infty} < \infty$), in which case we kill it.

As long as $\E{J_{\infty}} < \infty$, it is standard from renewal theory that one can create a stationary version of this process by the following procedure.  First generate a size-biased version $J_{\infty}^*$ of $J_{\infty}$.  Given that $J_{\infty}^* = t$, generate a path $((T^{(t)}_0(s),\rho), 0 \le s < t)$ which has the same law as $((T_0(s),\rho), 0 \le s < J_{\infty})$ conditioned on $J_{\infty} = t$.  Then finally take an independent $U[0,1]$ random variable $U$ and define $T(s) = T_0^{(t)}(Ut + s)$ for $s < (1-U)t$.  For $s \ge (1-U)t$, simply concatenate independent copies of $((T_0(s), \rho), 0 \le s < J_{\infty})$ onto the end to yield a path $((T(s), \rho), s \ge 0)$.

It turns out that there is a unique law $\cW$ on the ``environment'' of rooted trees that we aggregate onto $T(t)$ such that $\E{J_{\infty}} < \infty$ and also $T(0) \sim \cW$ (since $(T(t), t \ge 0)$ is stationary, this is also the law of $T(t)$ for any $t > 0$).  This law is awkward to describe fully, but it has the property that if $(T,\rho) \sim \cW$ then
\[
\p{|T| = k} = \frac{2}{k} \binom{2k-2}{k-1} 4^{-k}.
\]
Moreover, conditionally on $|T|=k$, if $v$ is picked uniformly from among the $k$ vertices of $T$ then $(T,v)$ has the same distribution as $(T,\rho)$ (i.e.\ $T$ is invariant under random re-rooting).  

Much easier to describe is the distribution of the \emph{genealogical tree} $G(t)$ associated with $T(t)$ via the aggregation process. 
This is an analogue of the cut-tree, where rather than thinking about edge-removal causing fragmentation we have edge-addition causing coalescence.  
For this it is useful to imagine an enriched version of the above process, in which the edges of the sampled trees are also marked with ``arrival times". The correct distribution for these marks may be deduced from the construction of $T(t)$. 

The genealogical tree $G(t)$ is a binary tree whose leaves correspond to vertices of $T(t)$ and whose internal vertices correspond to edges of $T(t)$. 
The root of $G(t)$ corresponds to the most recent edge to have appeared in $T(t)$. The two subtrees hanging off the internal vertex corresponding to an edge $e$ are the genealogical trees of the two clusters which were joined together by $e$. 

The stationarity of $T(t)$ induces stationarity for $G(t)$ and, in particular, for all $t$, $G(t)$ has the law of a critical binary Galton--Watson tree. Indeed, for a given $G$ with $k-1$ internal nodes and $k$ leaves, we have
\[
\p{G(t) = G} = \left(\frac{1}{2}\right)^{2k-1}
\]
and there are
\[
\frac{1}{k} \binom{2k-2}{k-1}
\]
such trees $G$.

How does one obtain the tree $T$ from its genealogical tree $G$?  Once again we need to mark the internal vertices of $G$ with the labels of the edges to which they correspond, after which we perform the reconstruction precisely as described in Section~\ref{subsec:general} for the cut-tree.  Moreover, because of the re-rooting invariance of a tree sampled according to $\cW$, it turns out that the two end-points of the edge marking a particular internal vertex of $G$ are uniformly distributed among the leaves of $G$ in the two subtrees of $G$ hanging off that internal vertex.  

Conditional on having $k$ leaves, $G(t)$ converges in distribution in the Gromov--Hausdorff--Prokhorov sense to a constant times the Brownian CRT, once its edge-lengths are rescaled by $k^{-1/2}$ and it is endowed with the uniform measure \cite{Kortchemski,Rizzolo}.  The law of the signposts in $G(t)$ is uniform on the relevant subtrees, which is precisely the discrete analogue of the law of the signposts in the Brownian CRT.  It is then natural to conjecture that, conditional on $|T(t)| = k$, a rescaled version of $T(t)$ also converges in distribution  to the Brownian CRT.  There are at least two proofs of this fact due to Edward Crane~\cite{Crane}; if an appropriately defined reconstruction map were known to be stable, this would provide a computation-free proof of the same result.

\section{Acknowledgements}

C.G.'s research was supported in part by EPSRC grant EP/J019496/1.  We are very grateful to Edward Crane for sharing the results of \cite{Crane} with us.

\bibliographystyle{abbrv}

\end{document}